\RequirePackage{fix-cm}
\documentclass{svjour3}                     
\smartqed  
\usepackage{tablefootnote}
\usepackage{multirow}
\usepackage{graphicx}
\usepackage{float}

\usepackage{amsmath}
\usepackage{graphics}
\usepackage{graphicx}
\usepackage{amssymb}
\usepackage{amsthm}
\usepackage{float}
\usepackage{color}
\usepackage{picture}
\usepackage[table,xcdraw]{xcolor}
\definecolor{mygray}{gray}{0.8}
\usepackage{multirow}
\usepackage{multicol}
\usepackage{mathtools}
\DeclarePairedDelimiter\ceil{\lceil}{\rceil}
\DeclarePairedDelimiter\floor{\lfloor}{\rfloor}

\usepackage[english,ruled,lined,linesnumbered]{algorithm2e}
\usepackage[noend]{algpseudocode}
\SetKwInput{KwData}{input}
\SetKwInput{KwResult}{output}
\makeatletter
\newcommand{\vast}{\bBigg@{4}}
\newcommand{\Vast}{\bBigg@{5}}
\makeatother

\usepackage[authoryear,round]{natbib}
\usepackage{blkarray}
\allowdisplaybreaks
\usepackage{footnote}

\journalname{}
\begin{document}

\title{Integrated Cutting and Packing Heterogeneous Precast Beams Multiperiod Production Planning Problem\thanks{This work was developed while the corresponding author was a master degree student at Federal University of Ceara}
}


\author{Kennedy Araujo  \and 
        Tiberius Bonates \and 
        Bruno Prata
}

\institute{K. Ara\'ujo \at
              Department of Applied Mathematics \\ University of Sao Paulo\\
              \email{kennedyanderson94@gmail.com}
           \and
           B. Prata \at
              Department of Industrial Engineering \\ Federal University of Cear\'a \\
            \email{baprata@ufc.br}
        \and
        T. Bonates \at 
        Department of Statistics and Applied Mathematics \\ Federal University of Cear\'a\\
          \email{tb@ufc.br}
}

\date{Received: date / Accepted: date}

\maketitle

\begin{abstract}
We introduce a novel variant of cutting production planning problems named Integrated Cutting and Packing Heterogeneous Precast Beams Multiperiod Production Planning (ICP-HPBMPP). We propose an integer linear programming model for the ICP-HPBMPP, as well as a lower bound for its optimal objective function value, which is empirically shown to be closer to the optimal solution value than the bound obtained from the linear relaxation of the model. We also propose a genetic algorithm approach for the ICP-HPBMPP as an alternative solution method. We discuss computational experiments and propose a parameterization for the genetic algorithm using D-optimal experimental design. We observe good performance of the exact approach when solving small-sized instances, although there are difficulties in finding optimal solutions for medium and large-sized problems, or even in finding feasible solutions for large instances. On the other hand, the genetic algorithm could find good-quality solutions for large-sized instances within short computing times.
\end{abstract}

\section{Introduction}

Nowadays, concrete precast production is increasingly trending in constructions sites. There are great advantages of using such kind of production, such as better and cheaper elements, and a potential to severely shorten construction time as compared to conventional methods. The precast element we consider in this work is a concrete precast beam, which is a kind of beam that is cast in plants away from the construction site, in a controlled environment.

These beams are heterogeneous in the sense that they can vary with respect to curing time, length and the number of traction elements used. We refer to the problem of planning the production of such beams to fulfill the clients demand within a given time horizon as the Heterogeneous Precast Beams Multiperiod Production Planning Problem (HPBMPP).

\cite{araujo2017} proposed four integer programming models for the HPBMPP, considering prestressed precast beams instead of conventional concrete precast beams. 
One of the proposed models minimizes the total idle capacity in the molds along the time horizon, two models to minimize the production makespan and one model for total completion time minimization. The authors also proposed several solution methods, in particular a size reduction heuristic that succeeded in finding high-quality solutions in shorter time and using less memory compared to exact methods. 

In this work, we propose a variant model of the HPBMPP, which consists in the integration of the production of bars, which are used in the precast beam production, into the problem. We divide the bars in two groups: \textit{standard bars} and \textit{leftovers}. Standard bars are new bars of standardized lengths, and leftovers are a type of bar that cannot be readily used in the beam production but can be stored in stock to produce other bars in the future. In this study, we consider that both standard bars and leftovers vary with respect to length. The production of bars to be used in the beam production can be made by the cutting of standard bars or leftovers in stock, or by the process of cutting overlapping leftovers. The overlapping process, consists in merging two or more leftovers to create a larger bar that can be cut to produce a bar of appropriate length that will be used in beam production. In this work, we only consider overlapping of two bars. To the best of our knowledge, the consideration of overlapping bars has not been previously studied.

We consider the integration into a single production planning problem of the cutting process of bars, or of overlapping bars, which must be packed in the molds for the production of a given demand of beams. We refer to this problem as the Integrated Cutting and Packing Heterogeneous Precast Beams Multiperiod Production Planning Problem (ICP-HPBMPP). Note that in this work we consider beams that are not prestressed. The mathematical model we propose is based on the model by \cite{arenales2015new}, which deals with the cutting stock/leftover problem, and on the model by \cite{araujo2017} for the HPBMPP. We consider that the bars needed to supply the beam production can be produced by cutting bars or leftovers in stock or by overlapping leftovers in stock. The stock is static, i.e., we are given an initial stock that is not replenished over the entire time horizon.

The ICP-HPBMPP is of practical interest because optimizing the production of prestressed beams has the potential effect of speeding up overall construction time, while improving the usage of molds and bar stock, while minimizing bars loss. An economical usage of bar stock may result in a reduction of unused bars in the construction site, which can improve the production flow. Furthermore, the reduction of concrete and bar loss may lead to a positive impact in the environment. An optimized process allows factories to accept additional orders due to shorter lead times. Also, the production cost with an optimized process will be lower, which may lead to a reduction of the final product's price, increasing competitiveness.

It is argued in \citep{araujo2017} that the HPBMPP is NP-hard since it includes, as a particular case, the classical one-dimensional cutting stock problem. Thus, the HPBMPP can become too difficult to solve as the dimension of instances increases. The computational results reported in Section \ref{P2:tests} show that the ICP-HPBMPP can be difficult to solve to optimality, justifying the use of decomposition techniques and heuristic procedures to deal with the problem. This also suggests that the HPBMPP is interesting to be studied from a theoretical point of view.

The remainder of this paper is organized as follows. In Section \ref{cwpLiterature} we discuss the literature of similar problems to the ICP-HPBMPP. In Section \ref{P2:Chap_Problem} we formally define the problem, propose an integer linear programming model for its solution, argue about its NP-hardness and propose a lower bound for its optimal objective function value. In Section \ref{P2:ChapPG} we present three constraint programming models for the generation of packing, cutting and overlapping patterns. In Section \ref{P2:GA} we propose a genetic algorithm for the problem under study. In Section \ref{P2:tests} we discuss several computational experiments conducted based on instances generated artificially and discuss the results of the proposed solution methods. In Section \ref{P2:conclusion} we discuss the conclusions and contributions of this chapter, as well as point out research gaps and suggest future work.

\section{Literature review}
\label{cwpLiterature}

To the best of our knowledge ICP-HPBMPP is not defined in the literature, even though the problem has similarities with one-dimensional cutting stock problems (1DCSP) and one-dimensional packing problems (1DPP). On the order hand, 1DCSP, 1DPP, and their variants have been substantially studied in the literature.

As far as one-dimensional cutting and packing problems (C\&P) are concerned, the studies of \cite{gilmore1961linear} and \cite{gilmore1963linear} proposed a column generation algorithm to solve the linear relaxation of large instances of 1DCSP. Such studies served as basis for a number of subsequent works. \cite{stadtler1990one} studied the 1DCSP proposing a heuristic based on the solution of the linear relaxation supplemented by a one-pass branching up procedure. The authors validated the proposed heuristic approach, testing on benchmark instances and on a case of study. \cite{dyckhoff1990typology} introduced a typology of C\&P problems, unifying notions in the literature to guide further research on particular types of those problems. \cite{vance1998branch} proposed two different branch-and-price approaches to find optimal solutions to the 1DCSP. \cite{wascher2007improved} presented a new typology to categorize the types of C\&P problems in the literature between years 1995 and 2004, introducing new categorization criteria. \cite{trkman2007one} proposed a model for the multiperiod one-dimensional cutting stock problems (M1DCSP), considering the use of objects/leftovers in stock. \cite{poldi2010problema} proposed an integer linear model for the M1DCSP, implemented a column generation to solve the linear relaxation, and developed two rounding heuristics for finding integer solutions to the problem. \cite{melega2018classification} proposed a mathematical model for the general integrated lot-sizing and cutting stock problem, and performed a vast classification of the literature of that problem, providing directions for future research.

Regarding the C\&P problems and optimization approaches in precast production, \cite{de2007using} described the problem of minimizing production costs for slabs of precast prestressed concrete joists and introduced a genetic algorithm to solve it. \cite{de2015integer} proposed an integer linear programming model for multiperiod production planning of precast concrete beams, which can be seen as a special case of the HPBMPP. \cite{arenales2015new} introduced a mathematical model for the cutting stock/leftover problem and suggested a column generation technique for finding the problem's linear relaxation solution. \cite{vassoler2017} proposed a mathematical model based on multiperiod cutting stock problem for the production planning problem of joists in trusses slabs industries. The authors suggested a solution method based on column generation to solve the linear relaxation of the problem. \cite{araujo2017} proposed several integer linear programming models for the Heterogeneous Prestressed Precast Beams Multiperiod Production Planning Problem, showed its NP-hardness and suggested a constraint programming model for generating cutting patterns for the problem. The authors also carried out computational experiments to validate the performance of the integer linear programming models. \cite{wang2018framework} introduced a two-hierarchy simulation-genetic algorithm hybrid model for precast production to ensure the on-time delivery of precast components minimizing the production cost, while simultaneously optimizing the resource waste under uncertainty in the processing time of each operation. The authors validated the model with a case study. 

The problem which we study in this work is the integration of the cutting stock/leftover problem proposed by \cite{arenales2015new} and the HPBMPP introduced by \cite{araujo2017}. We explore its solution via exact methods and heuristics methods in the case where instances cannot be solved by the state-of-art solvers.

\section{Problem statement}
\label{P2:Chap_Problem}

In this section we formally define the ICP-HPBMPP and propose an integer linear programming model for its solution based on the models proposed by \cite{arenales2015new} for the Cutting Stock/Leftover Problem (CSLP) and \cite{araujo2017} for the Heterogeneous Prestressed Precast Beams Multiperiod Production Planning Problem (HPPBMPP).

The ICP-HPBMPP consists in finding a feasible production planning to cast certain quantities of prestressed precast concrete beams, possibly of different types, while minimizing the total length of pieces of bars that cannot be used as \textit{leftover}. A leftover is understood here as a piece of bar that can be cut or overlapped in the future to meet new demands and is not considered waste. The beam factory has a fixed amount of bars and bar leftovers with standard lengths in stock that can be used within a given time horizon.

Each mold can only be used to cast one type of beam at a time. It is possible, however, to simultaneously cast beams of different lengths in the same mold, as long as they are of the same type. The total length of the beams produced during a given period in a given mold cannot be greater than the mold's capacity, and the total number of days required to complete the entire production cannot be greater than a given time horizon. After the process of cutting the bars, they are packed in the molds in order to produce the beams, note that different beam types can demand different numbers of bars. For this reason, we refer to this problem as a Cutting and Packing problem. The ICP-HPBMPP process can be seen in Figure \ref{flowchart}.

\begin{figure}[H]
    \centering
    \caption{Cutting and packing production flowchart}
    \includegraphics[width=.9\textwidth]{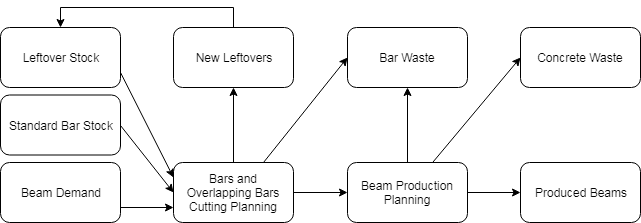}
    \label{flowchart}
\end{figure}

As input of the problem we have a deterministic static demand of beams, with their respective types and lengths, stock of bars and stock of bars leftovers, with their respective lengths. The cutting planning of bars is made for the entire time horizon, resulting in more bars leftovers (which can be used in another production planning), and, possibly, incurring in bar loss. The bars cut will be packed in the molds for the beam production along the given time horizon. After the production of all beams demanded is met, there will usually be concrete waste of the beams and additional loss of bars.

\subsection{Integer linear programming model}

In order to define a model for the ICP-HPBMPP, we make use of the same parameters defined in \citep{araujo2017}, as follows:

\begin{itemize}
    \item $M$: number of molds in which the beams are produced;
    \item $T$: number of available periods to complete the production;
    \item $C$: number of beam types;
    \item $q_c$: number of distinct lengths of beams of type $c$, with $ c = 1,\ldots,C$;
    \item $l(c,1), \; \ldots, \; l(c,q_c)$: real numbers corresponding to the actual lengths of beams of type $c$, with $ c = 1,\ldots,C$;
    \item $d(c,k)$: demand for beams of type $c$ and length $l(c,k)$, with $c = 1, \ldots,C$ and $k = 1,\ldots,q_c$;
    \item $t_c$: integer number corresponding to the curing time (in terms of periods) of beams of type $c$, for $c=1,\ldots,C$;
    \item $L_m$: real number corresponding to the capacity of mold $m$, with $m = 1,\ldots, M$;
    \item $P_i = (c_i, (a^i_1,\ldots,a^i_{q_{c_i}}))$: packing pattern, where $c_i$ stands for the beam type associated with pattern $P_i$ and $a^i_1,\ldots,a^i_{q_{c_i}}$ represent the quantity of each beam of length $l(c_i,1), \; \ldots, \; l(c,q_{c_i})$ in patterns $P_i$, with $i = 1, \ldots, r$, $c_i = 1, \ldots, C$. Note that $r$ represents the number of packing patterns;
    \item $P_0$: special pattern, which is used to denote that a mold is currently being used for the casting of a pattern that began in a previous period and whose production extends at least up to the current period.
\end{itemize}

Note that an idle mold (in other words, a mold that is not being used during a specific period) is not assigned the pattern $P_0$. In fact, it has no pattern assigned to it. 

In order to refer to specific information on a given pattern $P_i = \left({c_i},({a}_1,\ldots,{a}_{q_{c_i}})\right)$, we define the following notation:

\begin{itemize}
    \item $\mathcal{N}_i(c,k)$: number of beams of type $c$ and length $l(c,k)$ that pattern $P_i$ includes. If $c = {c_i}$, then $\mathcal{N}_i(c,k)={a}_k$, with $k \in \{1,\ldots,q_{c_i}\}$; otherwise, $\mathcal{N}_i(c,k) = 0$, for any $k$.
    \item $u(P_i)$: capacity used by $P_i$, i.e. $\displaystyle u(P_i) = \sum_{k=1}^{q_{{c_i}}} l({c_i},k) \cdot P_i({c_i},k)$, with $i=1,\ldots,r$.
    \item $E_i$: number of periods required to produce the beams in $P_i$, with  $i=1,\ldots,r$. This number equals the quantity of consecutive periods in which $P_i$ remains occupying a mold and is precisely the curing time of beams of type ${c_i}$, given by $t_{{c_i}}$.
\end{itemize}

Given a set of patterns $\mathcal{P} = \{P_1, \ldots, P_r\}$, not including $P_0$, we define some important sets as follows:

\begin{itemize}
    \item $Q(m)$: set containing the indices of the patterns in $\mathcal{P}$ whose capacity does not exceed the capacity of the $m$-th mold: $Q(m) = \{i \in \{1,\ldots,r\} : u(P_i) \leq L_m\}$, for $m=1,\ldots,M$. Note that the same pattern can belong to $Q(m)$ and $Q(m')$, with $m$ and $m'$ being two different molds of potentially distinct lengths.
    \item $Q^{\star}(m) = Q(m) \cup \{0\}$;
    \item $S(j)$: set of indexes of the patterns that have curing time $j \in \{1,..., R\}$, with $R = \max \{t_c: c = 1, \ldots, C\}$ being the largest curing time of all beam types present in the problem instance.
\end{itemize}

In what follows, we present the parameters that concern bars and bars leftover: 

\begin{itemize}
    \item $W$: number of different bar lengths;
    \item $V$: number of different bar leftover lengths;
    \item $H$: number of cutting patterns;
    \item $O$: number of overlapping patterns;
    \item $\Gamma$: number of different mold lengths;
    \item $b_{1},\ldots,b_{W}$: bar lengths;
    \item $b_{W+1},\ldots,b_{W+V}$: bar leftover lengths allowed. Note that this data narrows the types of cutting, and overlapping patterns;
    \item $\mathcal{L}_1,\ldots,\mathcal{L}_{\Gamma}$: mold lengths. Note that this data narrows the types of cutting, and overlapping patterns;
    \item $G(\mathcal{L}_{\gamma})$ = set of molds which are of length $\mathcal{L}_{\gamma}$, $\gamma = 1, \ldots \Gamma$;
    \item 
    $H_w$: set of cutting patterns for bar of length $b_w$ that do not include leftovers.
    \item 
    $H_w(v)$: set of cutting patterns for bar type $w$ that include leftovers of length $b_{W+v}$;
    \item 
    $\mathbb{O}$: set of overlapping patterns;
    \item 
    $\mathbb{O}(\gamma)$: set of overlapping patterns that produce bars of length $\mathcal{L}_{\gamma}$.
    \item 
    $I_h = (w_h, (a^h_1,\ldots,a^h_{\Gamma},a^h_{\Gamma+1},\ldots,a^h_{\Gamma + V}))$: cutting pattern used to cut a bar of index $w_h = 1, \ldots, W + V$, with $h = 1, \ldots, H$. Note that $a^h_1,\ldots,a^h_{\Gamma}$ are the number of bars of lengths $\mathcal{L}_1,\ldots,\mathcal{L}_{\Gamma}$ and $a^h_{\Gamma+1},\ldots,a^h_{\Gamma+V}$ are the number of bars of lengths $b_{W+1},\ldots,b_{W+V}$;
    \item 
    $\mathbb{O}_{\mu} = (\gamma_{\mu},( a^{\mu}_1,\ldots,a^{\mu}_V))$: overlapping pattern that generates a bar of length $\mathcal{L}_{\gamma_{\mu}}$, with $\gamma_{\mu} = 1, \ldots, \Gamma$ and $\mu = 1, \ldots, O$. Note that $ a^{\mu}_1,\ldots,a^{\mu}_V$ are the number of bars of lengths $b_{W+1},\ldots,b_{W+V}$;
    \item 
    $D_{c_i}$ = number of bars that a pattern $P_i$ with beam type $c_i$ demands;
    \item $e_w$ = number of bars of length $b_w$ in stock, leftover or otherwise, with $w = 1,\ldots,W+V$;
    \item $a_{v,\mu}$ = number of leftovers of length $b_{W+v}$ in overlapping pattern $\mathbb{O}_{\mu}$, with $\mu = 1, \ldots, O$.
    \item $a_{\gamma,h,w}$ = number of objects of length $\mathcal{L}_{\gamma}$ cut from a bar of length $b_w$ following a cutting pattern $I_h$ that generates no leftover, with  with $w = 1,\ldots,W+V$;
    \item $a_{\gamma,h,w,v}$ = number of objects of length $\mathcal{L}_{\gamma}$ cut from a bar of length $b_w$ following a cutting pattern $I_h$ that generates a leftover of length $b_{W+ v}$, with $w = 1,\ldots,W$ and $v = 1, \ldots, V$.
    \item $f_{h,w}$ = waste resulting from using a cutting pattern $I_h$ to cut a bar of length $b_w$ generating no leftover, with $w = 1,\ldots,W+V$.
    \item $f_{h,w,v}$ = waste resulting from using a cutting pattern $I_h$ to cut a bar  of length $b_w$ generating a leftover of length $b_{W+v}$, with $w = 1,\ldots,W$ and $v = 1, \ldots, V$.
    \item $f_{\mu}$ = waste of bar produced by overlapping pattern  $\mathbb{O}_{\mu}$, with $\mu = 1,\ldots,O$.
\end{itemize}

\noindent We present the decision variables below:

\begin{itemize}
    \item[] \[
x_i^{m,t} = \left\{ \begin{array}{rl}
         1, & \mbox{if the packing pattern $P_i$ starts to be used in} \\ & \mbox{mold $m$ at period $t$ (and its usage, naturally,} \\ & \mbox{lasts for $E_i$ periods)}; \\
         0, & \mbox{otherwise.}\end{array}\right.
\]
    \item[]\[
z_{t} = \left\{ \begin{array}{rl}
         1, & \mbox{if as least one mold is used at period $t$, for $t = 1, \ldots, T$;}\\
         0, & \mbox{otherwise.}\end{array}\right.
\]
    \item[] $y_{h,w}$: number of bars of length $b_w$ cut following a cutting pattern $I_h \in H_w$.
    \item[] $y_{h,w,v}$: number of bars of length $w$ cut following a cutting pattern $I_h \in H_w(v)$ generating a leftover of length $b_{W+v}$.
    \item[] $o_{\mu}: $ number of times the overlapping pattern $\mathbb{O}_{\mu}$ was used,  ${\mu} \in \mathbb{O}$.
\end{itemize}

Note that variables $y_{h,w}, y_{h,w,v}$, and $o_{\mu}$ are nonnegative integer decision variables. We present the integer linear programming model proposed for the ICP-HPBMPP as follows:

\begin{flalign}
\mathclap{\textbf{(ICP) min}} \nonumber\\
& \lambda_1 \sum_{t = 1}^T z_t +
\lambda_2 \sum_{w=1}^W \sum_{h \in H_w} f_{h,w}y_{h,w} & \nonumber
\\
& +
\lambda_3 \sum_{w=1}^W \sum_{v=1}^V \sum_{h \in H_w(v)} f_{h,w,v}y_{h,w,v} \nonumber\\
& + \lambda_4 \Bigg(\sum_{w=W+1}^{W+V} \sum_{h \in H_w}  f_{h,w}y_{h,w} +  \sum_{\mu \in \mathbb{O}} f_{\mu}o_{\mu} \Bigg) & \qquad \qquad \label{FOCWP}
\end{flalign}

\begin{flalign}
\mathclap{\textbf{s. t.}}  \nonumber
\\
& \sum_{i \in Q^{\star}(m)}x_i^{m,t} \leq 1, &\; m=1,\ldots,M, \; t=1,\ldots,T  \label{P2:restr1porforma}
\\
& \sum_{m=1}^M  \sum_{i \in Q(m)} \sum_{t=1}^{T - E^i + 1} P_i(c,k) \: x_i^{m,t} \geq d(c,k), & \;c=1,\ldots,C,\; k=1,\ldots,q_c  \label{P2:restrdemanda}
\\
& (E_i - 1) \; x_i^{m,t} \leq \sum_{\alpha = 1}^{E_i - 1} \; x_0^{m,t + \alpha}, & \; m=1,\ldots,M, \nonumber \\ && t=1,\ldots,T - E_i + 1, \nonumber \\ && i \in Q(m) \label{P2:sequenciamento1}
\\
&x_0^{m,1} = 0 & m = 1,\ldots, M, \label{P2:sequenciamento3}
\\
&x_0^{m,t} \leq \sum_{\gamma = 2}^R \; \sum_{j = \gamma}^R \; \sum_{i \in \{Q(m) \cap S_j  \} }x_i^{m,t - \gamma + 1}, & \; m=1,\ldots,M, \; t=2,\ldots,T  \label{P2:sequenciamento2}
\\
&M \; z_t \geq \sum_{m = 1}^M \left( \sum_{i \in Q^*(m)}\; x_i^{m,t} \right), & t = 1,\ldots, T\label{P2:z_r}
\\
&\sum_{i \: \in \: Q^*(m) }  x_i^{m,t} \geq \sum_{i \: \in \: Q^*(m) } x_i^{m, t+1}, & m = 1,\ldots, M, \;  t = 1,\ldots, T-1 \label{P2:continuidade} 
\\
& \sum_{h \in H_w} y_{h,w} + \sum_{\mu \in \mathbb{O}} a_{w,{\mu}}o_{\mu} \leq e_w, \qquad& w = W+1,\ldots,W+V \label{CQP1}
\\
& \sum_{h \in H_w} y_{h,w} + \sum_{v=1}^V \sum_{h \in H_w(v)} y_{h,w,v} \leq e_w,& w = 1,\ldots,W \label{CQP2}
\end{flalign}

\vspace{-20pt}

\begin{flalign}
& \sum_{w=1}^{W+V} \sum_{h \in H_w} a_{\gamma,h,w}y_{h,w} + 
 \sum_{w=1}^W \sum_{v=1}^V \sum_{h \in H_w(v)}  a_{\gamma,h,w,v}y_{h,w,v} \nonumber\\ & + \sum_{\mu \in \mathbb{O}(\gamma)} o_{\mu} = \sum_{m \in G(\mathcal{L}_{\gamma}) } \sum_{t=1}^T \sum_{i \in Q(m)} D_{c_i} x_{i}^{m,t}, \quad & \gamma = 1, \ldots, \Gamma \label{CQP3}
 \end{flalign}
 
\vspace{-15pt}
 
 \begin{flalign}
&x_i^{m,t} \in \{0,1\}, &m=1,\ldots,M,\; t=1,\ldots,T, \; i \in Q^*(m) \label{P2:varsdecisao}
\\
& z_t \in \{0,1\}, & t = 1,\ldots,T
\\
&y_{h,w} \in \mathbb{Z}_+,& w = 1,\ldots,W,\;h \in H_w\\
&y_{h,w,v} \in \mathbb{Z}_+,& w = 1,\ldots,W, \; v = 1,\ldots,V, \; h \in H_w(v) \label{ybounds}\\
&o_{\mu}  \in \mathbb{Z}_+, &{\mu} \in \mathbb{O} \label{obounds}.
\end{flalign}

The objective function (\ref{FOCWP}) is divided into 4 terms. The first term is the makespan value. The second term defines the waste related to the use of new bars to produce the demand of bars. The third term describes the waste associated to the use of new bars to produce the bars required by beam production while creating new leftovers. Finally, the fourth term specifies the waste corresponding to the bar leftovers in stock that are used to produce the amount of bars required. Note that each term of (\ref{FOCWP}) could alternatively be regarded as an independent objective functions to be minimized. We obtain (\ref{FOCWP}) using the weighted sum method, in which the parameters $\lambda_i \in \mathbb{R_+}$, with $i=1,\ldots,4$, indicate the weight of each objective function term. A solution that minimizes (\ref{FOCWP}) is, therefore, a Pareto optimum~\citep{marler2004survey}.

Constraints (\ref{P2:restr1porforma}) ensure that at most one pattern must be assigned to mold $m$ at period $t$, with the possibility of this pattern being $P_0$. Constraint set (\ref{P2:restrdemanda}) requires that all demands must be satisfied. Constraints (\ref{P2:sequenciamento1}) force that, if pattern $P_i$ is initiated at period $t$, then the next $E_i - 1$ periods shall have the pattern $P_0$ assigned to them (the right-hand side of the constraint remains unconstrained, in case $x_i^{m,t}=0$). Constraint sets (\ref{P2:sequenciamento3}) and (\ref{P2:sequenciamento2}) establish that $P_0$ shall only be used in mold $m$ if there is some pattern associated with a previous period in the same mold, whose production has not yet been completed. 

Each constraint in set (\ref{P2:z_r}) ensures that variable $z_t$ must be 1 if period $t$ is used to produce beams. Constraints (\ref{P2:continuidade}) force that there is no inactive period during beam production in the molds. This means that the production is continuous, i.e., if a mold is used it will be used with no interruption; in other words, if the production stops at a given mold and period, it will not resume in that mold at a subsequent period.

Constraints (\ref{CQP1}) establish that the number of bar leftovers cut plus the number of leftover bars used to produced bars via overlapping does not exceed the stock, note that the cutting of a leftover does not generate leftovers. Constraint set (\ref{CQP2}) ensures that the number of bars cut does not exceed the stock. Constraints (\ref{CQP3}) force that the amount of bars necessary to produce the beams is achieved, assuming that the required amount of bars is the number of bars used by the forms in the entire time horizon. 
Constraints (\ref{P2:varsdecisao})-(\ref{obounds}) define the domains of the decision variables.

The model (ICP) has $\mathcal{O}(MTr + WVH + O)$ variables and $\mathcal{O}(q + MTr + V + W + \Gamma)$ constraints, with $ \displaystyle q = \sum_{i = 1}^C q_c$. Thus, depending on the total number of possible packing, cutting, and overlapping patterns, there may be an excessive number of variables and constraints in the model. We choose to limit the number of packing patterns, which are the more numerous type of pattern, in practice, by using only maximal packing patterns, used successfully by \citep{vance1998branch} and \citep{araujo2017}. We say that a pattern $P_i$ \textit{contains} a pattern $P_j$ if $c_i = c_j$ and $a^i_k \geq a^j_k$, with $k = 1, \ldots, q_{c_i}$.

\begin{proposition}
Restricting the model (ICP) to using only maximal packing patterns does not modify its set of optimal solutions.
\end{proposition}

\begin{proof}
Given an optimal solution to model (ICP) that is composed by non-maximal packing patterns we claim that replacing the non-maximal packing patterns with maximal ones that contain such patterns will not have an impact on the makespan. Indeed, the actual number of periods used to fulfill the demand will remain unaffected, given that all packing patterns of a given type have the same associated curing time. In the same way, there will be no changes to the cutting and overlapping patterns used in the optimal solution since the number of bars needed for the beam production will remain unchanged.
\end{proof}


\subsection{NP-hardness}
To argue the ICP-HPBMPP hardness note that for instances where $D_c = 0$, for all $c = 1, \ldots, C$, constraints (\ref{CQP1})-(\ref{CQP3}) are naturally fulfilled and all variables $y_{h,w}, y_{h,w,v}$ and $o_{\mu}$ are set to zero, reducing an instance of ICP-HPBMPP to an HPPMBPP instance involving the minimization of the makespan, up to a constant multiplicative factor. Consequently, the ICP-HPBMPP is a generalization of HPPMBPP, which is already known to be NP-hard \citep{araujo2017}.

\subsection{Objective function lower bound}

Since the ICP-HPBMPP is a NP-hard problem, a lower bound for the optimal objective function value may help in evaluating the quality of feasible solutions in heuristic and exact methods. In order to simplify the presentation of our proposed lower bound for objective function (\ref{FOCWP}) optimal value, we present the following notation. For a given $\gamma \in \{1, \ldots, \Gamma \}$ we define the following sets:

{
\begin{itemize}
    \item $C1_{\gamma} = \{f_{h,w}/a_{\gamma,h,w}: h \in H_w \mkern12mu\wedge \mkern12mu a_{\gamma,h,w} > 0 \mkern12mu\wedge \mkern12mu 1 \leq w \leq W \}$
    \item $C2_{\gamma} = \{ \alpha' f_{h,w,v}/a_{\gamma,h,w,v}: h \in H_w(v) \mkern12mu\wedge \mkern12mu a_{\gamma,h,w,v} > 0\mkern12mu\wedge \mkern12mu 1 \leq w \leq W \mkern12mu\wedge \mkern12mu 1 \leq v \leq V \}$
    \item $C3_{\gamma} = \{ \alpha'' f_{h,w}/a_{\gamma,h,w}: h \in H_w \mkern12mu\wedge \mkern12mu a_{\gamma,h,w} > 0 \mkern12mu\wedge \mkern12mu W + 1 \leq w \leq W + V\}$
    \item $C4_{\gamma} = \{ \alpha''f_{\mu}: \mu \in \mathbb{O}(\gamma)\}$
    \item $\hat{C}_{\gamma} = \{C1_{\gamma} \cup C2_{\gamma} \cup C3_{\gamma} \cup C4_{\gamma} \}$
\end{itemize}
}

An upper bound on the optimal value of model (ICP) is given by Equation (\ref{CPlowerbound}).

\begin{flalign}
   &\ceil*{{\displaystyle\sum_{c = 1}^C t_c \cdot \left( \sum_{k=1}^{q_c} l(c,k) \cdot d(c,k) \right)}/{\displaystyle\sum_{m =1}^M L_m}} 
   + \nonumber\\ &\min_{\gamma \in \{1, \ldots, \Gamma\}}\Bigg\{ \ceil*{{\displaystyle\sum_{c = 1}^C D_c \cdot  \left( \sum_{k=1}^{q_c} l(c,k) \cdot d(c,k) \right)}/{\mathcal{L}_{\gamma}}} \cdot \min\{\hat{C}_{\gamma}\} \Bigg\} \label{CPlowerbound}
\end{flalign}

The first part of Equation (\ref{CPlowerbound}) corresponds to a lower bound for the makespan, while the second part stands for the minimum waste resulting from using molds of some fixed length $\mathcal{L}_{\gamma}$.

\section{Patterns generation} \label{P2:ChapPG}
	
	Instead of carrying out exhaustive enumerations, we generated the desirable packing, cutting, and overlapping patterns for a given instance using constraint programming models, which are described in the remainder of this section.

\subsection{Packing patterns generation}
 Consider the following notation, in addition to the notation presented in Section \ref{P2:Chap_Problem}:

    \begin{itemize}
        \item $K$: the largest number of different lengths among beam types, i.e. $\max q_c$ with $c= 1, \ldots, C$. For example, in an instance with 2 beam types, in which type $1$ has $6$ distinct beam lengths and type $2$ has $4$ distinct beam lengths, we have $K=6$.
        \item $v_i \in \{1,\ldots,C\}$: a decision variable that corresponds to the type of beam used by the pattern $P_i$.
         \item $\gamma_i \in \{1,\ldots,\Gamma\}$: auxiliary decision variable for generating patterns that will be maximal in at least one mold of the problem. It defines in which mold capacity the generated pattern $P_i$ is maximal. 
        \item $A^i \in \mathbb{Z}^{K}$: a vector of decision variables, with $A_j$ representing the number of beams of the length $\ell(v,j)$, for all $j \in \{1,\ldots,K\}$. Given a pattern $P_i$ of type $v$, the nonzero components of vector $A^i$ correspond to $[\mathcal{N}_i(v,j)]_{j=1}^{q_v}$.
        \item $P_i = \left({v_i},({A^i_1},\ldots,{A^i_{q_v}})\right)$: the generated pattern.
    \end{itemize}

For the generation of a packing pattern $P_i$ we propose the model, which is adapted from \citep{araujo2017}.
    
    \begin{flalign}
        &1 \leq v_i \leq C,   \label{P2:ppr1}\\
        &1 \leq \gamma_i \leq \Gamma, \label{P2:ppr1.2}\\
        &A^i_j = 0, \text{ if } v_i = c,  &      &c = 1,\ldots,C, \nonumber \\
        & &&j = q_c+1,\ldots, K \label{P2:ppr2}\\
        &\mathcal{L}_{m} - \min_{j = 1, \ldots, q_c}(l(c,j)) < \sum_{j=1}^{q_c} l(c,j) \cdot A^i_j \leq \mathcal{L}_{m}, \text{ if } (v_i = c \land \gamma_i = m), &&c = 1,\ldots,C, \nonumber\\
         & &&m = 1, \ldots, \Gamma, \label{P2:ppr3}\\
        &A^i_k \in \mathbb{Z_+}, &  & \;  k = 1,\ldots, K. \label{P2:ppr4}
    \end{flalign}
    
Constraint (\ref{P2:ppr1}) implies that the pattern type has domain $\in \{1, \ldots, C \}$. Constraint (\ref{P2:ppr1.2}) defines the length of the molds in which the generated pattern should be maximal. Constraint set (\ref{P2:ppr2}) implies that if the generated pattern is of type $ v $ then it includes no beam of size $l(v,j)$, such that $ j> q_v $. Constraint set (\ref{P2:ppr3}) imposes   that the capacity used by the generated pattern is simultaneously larger than the mold length minus the shortest beam length from its type and no larger than the length of the actual mold. The empty pattern is, therefore, not generated and has to be manually included in the final set of patterns. We utilized the solver CPLEX CP Optimizer to enumerate all the solutions of model (\ref{P2:ppr1})-(\ref{P2:ppr4}).

\subsection{Cutting patterns generation}

In this section we propose a constraint programming model for cutting patterns generation. The decision variables are given below:

\begin{itemize}
    \item $w_{h}$: index of the bar that will be cut in the generated cutting pattern $I_h$;
    \item $A^h_i$: number of items of length $\mathcal{L}_i$ cut in the pattern, for $i \in \{1, \ldots, \Gamma\}$;
    \item $A^h_i$: number of items of length $b_{W+i}$ cut in the pattern, for $i \in \{\Gamma+1, \ldots, \Gamma+V\}$;
    \item $I_h = \left({w_h},({A^h_1},\ldots,{A^h_{\Gamma}},{A^h_{\Gamma+1},\ldots, {A^h_{\Gamma + V}} })\right)$: the generated pattern.
\end{itemize}

The proposed constraint model for generating a cutting pattern $H_h$ is given by Equations (\ref{PPR_P2_1})-(\ref{PPR_P2_5}).

\begin{align}
    &1  \leq w_{h}  \leq W + V, \label{PPR_P2_1}
    \\ 
    &\sum_{i = 1}^{\Gamma} \mathcal{L}_i \cdot A^h_i + \sum_{i = 1}^{V} b_{W+i} \cdot A^h_{\Gamma + i}  \leq \text{element}(w_{h},b), \label{PPR_P2_2}
    \\
    &\#\{i \in \{\Gamma + 1, \ldots, \Gamma + V \}| A^h_i > 0\} = 1, \label{PPR_P2_3}
    \\
    & A^h_i = 0, \text{ if } w_{h} > W, & & \; i = \Gamma + 1, \ldots, \Gamma + V \label{PPR_P2_4}
    \\
    &A^h_i \in \mathbb{Z_+},  & &\; i = 1,\ldots, \Gamma+V. \label{PPR_P2_5}
\end{align}

Constraint (\ref{PPR_P2_1}) defines the domain of each decision variables $w_h$. Each $w_h$ variable determines defines the bar that will be cut in the current pattern to generate items. If $ 1 \leq w \leq W $, the bar that will be cut is a new bar. If $ W + 1 \leq w \leq W+V$, the bar that will be cut is a bar leftover. Constraint (\ref{PPR_P2_2}) states that the total length of items cut in the pattern must be shorter than the length of the bar used to cut such pattern, with expression element($w_{h},b$) standing for the $w_{h}$-th element of array $b$ \citep{nicolasbeldiceanumatscarlsson2018}. Constraint set (\ref{PPR_P2_3}) implies that a cutting pattern only generates one type of leftover. Constraint (\ref{PPR_P2_4}) implies that a leftover does not generate more leftovers. We utilized the CPLEX CP Optimizer to enumerate all the solutions of model (\ref{PPR_P2_1})-(\ref{PPR_P2_5}).

\subsection{Overlapping patterns}

In order to enrich the problem by allowing the possibility of using overlapping bars, we recall that an overlapping pattern $\mathbb{O}_{\mu}$ is a tuple $\mathbb{O}_{\mu} = (\gamma_{\mu},( a^{\mu}_1,\ldots,a^{\mu}_V))$. Note that $\gamma$ is associated to the length of the bar that is generated in such pattern. Such length must be equal to the capacity of some mold, since we are only required to produce bars via overlapping that are used for beam production. A bar produced by overlapping is only produced from leftovers in stock.

In order to simplify the model's notation, consider the following decision variables:

\begin{itemize}
    \item $A^{\mu}_i$: decision variable that represents number of items $b_{W+i}$ used in the overlapping pattern, for $i \in \{1, \ldots, V\}$.
    \item $\gamma_{\mu} \in \{1,\ldots,\Gamma\}$: decision variable that defines the length of the bar produced by the overlapping pattern.
    \item $f \geq 0$: decision variable that expresses the waste of bar associated to the overlapping pattern to produce a bar of length $\mathcal{L}_{\gamma}$.
    \item $\mathbb{O}_{\mu} = \left({\gamma},({A^{\mu}_1},\ldots,{A^{\mu}_{V}})\right)$: the generated pattern.
\end{itemize}

The following constraint programming model can be used to produce an overlapping pattern:

\begin{flalign}
    &1 \leq \gamma_{\mu} \leq \Gamma, \label{sp_gamma}\\
    &\sum_{i = 1}^{V} A^{\mu}_i b_{W+i} \geq \mathcal{L}_{\gamma_{\mu}} + \epsilon, \label{capacity}\\
    & \sum_{i = 1}^{V} A^{\mu}_i = 2,\label{sp_2}\\
    & f = \mathcal{L}_{\gamma_{\mu}} - \sum_{i = 1}^{V} A^{\mu}_i b_{W+i}. \label{sp_waste}
\end{flalign}

Constraint (\ref{sp_gamma}) ensures that the length of the bar produced is one of the possible mold lengths. Constraint (\ref{capacity}) forces that the total length of the chosen leftovers is greater than the length of the bar produced via overlapping plus a constant $\epsilon$ which is the loss of the bar resulting from the overlapping process. Constraint (\ref{sp_2}) defines that only 2 leftovers are used in the production of the bar made via overlapping. Constraint (\ref{sp_waste}) defines the bar waste resulting from the overlapping pattern.

The constraint programming model for overlapping pattern generation is sufficiently flexible to accommodate the production planner's necessities. In a more general setting, we could require that a bar made via overlapping can only be produced by using more than 2 and no more than a predefined number of leftovers and specify the $\epsilon$ value to be proportional to the number of leftovers used in such pattern.

\section{Genetic algorithm for the ICP-HPBMPP}
\label{P2:GA}

In this section we propose a genetic algorithm to solve the ICP-HPBMPP, formalize the solution representation chosen, the solution fixing procedure, the selection, mutation, and crossover operators, as well as the initial population generation, population restart, and local search.

\subsection{Solution representation}

The solution representation consists of a 2-row matrix, in which each column $j$ consists of the genes $a_j$ and $x_j$, where $a_j$ is a pattern index and $x_j$ is the number of times the pattern represented by $a_j$ is used. The number of columns of this representation is variable and can be at most $r+H+O$. The $a_j$ genes can have values in $\{ 1,\ldots,r+H+O$\}, in which the values $1, \ldots, r$ represent the packing patterns indices, the values $r + 1, \ldots, r + H$ correspond to the cutting patterns indices, and the values $r + H + 1, \ldots, r + H + O$ are associated with the indices of overlapping patterns. In Figure \ref{fig:rep_gen}, we show a generic scheme of the solution representation, in which the number of columns is exactly $r+H+O$.

\begin{figure}[H]
    \centering
    \caption{Solution representation}
    \includegraphics[width=.4\linewidth]{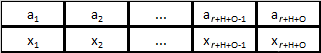}
    \label{fig:rep_gen}
\end{figure}


In order to illustrate the solution representation we first present instance cwp000, generated randomly, in Table \ref{tab:cwp000}. Its respective packing, cutting, and overlapping patterns are presented in Tables \ref{tab:cwp000p}, \ref{tab:cwp000c}, and \ref{tab:cwp000o}, respectively.

\begin{table}[H]
\centering
\caption{Instance cwp000 description}
\begin{tabular}{l} 
\hline
\multicolumn{1}{c}{\textbf{Instance cwp000}}  \\ 
\hline
$C =1 \qquad M = 5 \qquad T = 3$                        \\
$W = 1 \qquad V = 4$                               \\
$L = (5.95, 5.95, 5.95, 5.95, 11.95)$         \\
$t_1 = 1$                                     \\
$q_1 = 2$                                     \\
$D_1 = 1$                                     \\
$l(1,\cdot) = (1.12, 3.3)$                    \\
$d(1,\cdot) = (5, 10)$                        \\
$b = {(}12, 2, 5, 6, 8{)}$                    \\
$e = (30, 16, 28, 25, 29)$                    \\
$\epsilon = 0.3$                              \\
\hline
\end{tabular}
\label{tab:cwp000}
\end{table}


\begin{table}[H]
\caption{Packing patterns for instance cwp000}
\centering
\begin{tabular}{ccccc}
\hline
\textbf{ID} & \textbf{\begin{tabular}[c]{@{}c@{}}Beam\\ type\end{tabular}} & \textbf{Capacity} & \textbf{$a^p_1$} & \textbf{$a^p_2$} \\ \hline
1 & 1 & 5.6 & 5 & 0 \\
2 & 1 & 5.54 & 2 & 1 \\
3 & 1 & 11.2 & 10 & 0 \\
4 & 1 & 11.14 & 7 & 1 \\
5 & 1 & 11.08 & 4 & 2 \\
6 & 1 & 11.02 & 1 & 3 \\ \hline
\end{tabular}
\label{tab:cwp000p}
\end{table}


\begin{table}[H]
\caption{Cutting patterns for instance cwp000}
\centering
\begin{tabular}{ccccccccc}
\hline
\textbf{ID} & \textbf{\begin{tabular}[c]{@{}l@{}}Bar\\ cut\end{tabular}} & \textbf{Capacity} & $a^h_1$ & $a^h_2$ & $a^h_3$ & $a^h_4$ & $a^h_5$ & $a^h_6$ \\ \hline
7 & 1 & 5.95 & 1 & 0 & 0 & 0 & 0 & 0 \\
8 & 1 & 7.95 & 1 & 0 & 1 & 0 & 0 & 0 \\
9 & 1 & 9.95 & 1 & 0 & 2 & 0 & 0 & 0 \\
10 & 1 & 11.95 & 1 & 0 & 3 & 0 & 0 & 0 \\
11 & 4 & 5.95 & 1 & 0 & 0 & 0 & 0 & 0 \\
12 & 5 & 5.95 & 1 & 0 & 0 & 0 & 0 & 0 \\
13 & 1 & 11.95 & 1 & 0 & 0 & 0 & 1 & 0 \\
14 & 1 & 10.95 & 1 & 0 & 0 & 1 & 0 & 0 \\
15 & 1 & 11.9 & 2 & 0 & 0 & 0 & 0 & 0 \\
16 & 1 & 11.95 & 0 & 1 & 0 & 0 & 0 & 0 \\ \hline
\end{tabular}
\label{tab:cwp000c}
\end{table}


\begin{table}[H]
\caption{Overlapping patterns for instance cwp000}
\centering
\begin{tabular}{ccccccc}
\hline
\textbf{ID} & \textbf{\begin{tabular}[c]{@{}c@{}}Bar\\ generated\end{tabular}} & \textbf{\begin{tabular}[c]{@{}l@{}}Waste \\ of bar\end{tabular}} & \textbf{$a^{\mu}_1$} & $a^{\mu}_2$ & $a^{\mu}_3$ & $a^{\mu}_4$ \\ \hline
17 & 1 & 1.05 & 1 & 1 & 0 & 0 \\
18 & 1 & 4.05 & 0 & 2 & 0 & 0 \\
19 & 1 & 2.05 & 1 & 0 & 1 & 0 \\
20 & 1 & 6.05 & 0 & 0 & 2 & 0 \\
21 & 1 & 5.05 & 0 & 1 & 1 & 0 \\
22 & 1 & 8.05 & 0 & 0 & 1 & 1 \\
23 & 1 & 4.05 & 1 & 0 & 0 & 1 \\
24 & 1 & 7.05 & 0 & 1 & 0 & 1 \\
25 & 1 & 10.05 & 0 & 0 & 0 & 2 \\
26 & 2 & 4.05 & 0 & 0 & 0 & 2 \\
27 & 2 & 1.05 & 0 & 1 & 0 & 1 \\
28 & 2 & 2.05 & 0 & 0 & 1 & 1 \\ \hline
\end{tabular}
\label{tab:cwp000o}
\end{table}

Note that ID is associated with the pattern indices. An optimal solution for the cwp000 instance is shown as the chromosome in Figure \ref{fig:solucao_numerica}.

\begin{figure}[H]
    \centering
    \caption{Example of a feasible solution of instance cwp000}
    \includegraphics{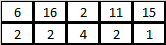}
    \label{fig:solucao_numerica}
\end{figure}

For the solution in Figure \ref{fig:solucao_numerica} we obtain an objective function value of 2.1, with makespan of 2 periods and bar waste of 0.1m. Figure \ref{fig:gannt_numerica} shows that packing patterns with indices 2 and 6, were used 4 and 2 times, respectively. Due to the fact that we are restricted to using only maximal packing patterns in their respective molds and a given packing pattern is maximal with respect to only one distinct length of mold, we infer that packing pattern 2 is associated with molds of length 5.95m, and packing pattern 6 is associated with molds of length 11.95m. Therefore, we need to produce a total number of 2 bars of length 5.95m and 6 bars of length 11.95, since the beam type produced by each solution packing patterns requires only one bar. The cutting patterns used are those with indices 11, 15 and 16, and their frequencies are 2, 1, and 2, respectively. None of the overlapping patterns was selected in the solution.

The production planning consists of the specification of the exact quantity of bars required for the beam production as long as the available stock of bars is not violated. Thus, the solution represented encoded in the chromosome in Figure \ref{fig:solucao_numerica} is feasible.

\begin{figure}[H]
    \centering
    \caption{Gantt chart for an optimal solution of instance cwp000}
    \includegraphics[width=.5\linewidth]{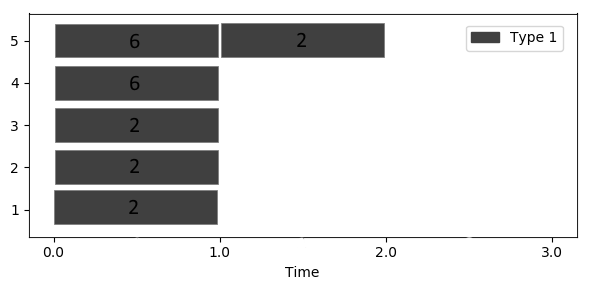}
    \label{fig:gannt_numerica}
\end{figure}

\subsection{Initial population generation}

Since we typically need a large quantity of individuals to generate a population, deterministic methods are not the best choice, despite the high-quality solutions produced by them. We propose a pseudorandom approach to generate a large quantity of solutions, which is described in Algorithm \ref{gerar_ind}.

\begin{algorithm}[H]
 \KwData{Instance, Set of Packing Patterns, Set of Cutting Patterns, Set of Overlapping Patterns}
 \KwResult{Feasible solution}
    Initialize $solution$ with all patterns with their respective frequencies set to zero.\\
    \While{Beam demands is not fulfilled}{
        $pac_p \leftarrow$ random packing pattern that has not yet been selected.\\
        \If{There is some beam in $pac_p$ whose demand is unfulfilled}{
        Increment the number of times that $pac_p$ is used in $solution$ until all beams in $pac_p$ have their demands fulfilled.\\
        }
    }
    Calculate the number of bars needed according to the packing patterns frequencies\\
    \For{each mold length $\gamma$}{
        \While{(number of bars of length $\mathcal{L}_{\gamma}$ needed was not reached) $\lor$ (there is at least one cutting pattern not selected)}{
            $cut_p \leftarrow$ random cutting pattern that generates bars of length $\mathcal{L}_{\gamma}$ that has yet not been selected.\\
            
            bars\_needed $\leftarrow$ number of bars of length $\mathcal{L}_{\gamma}$ required.\\
            n $\leftarrow$ number of times $cut_p$ can be added to $solution$ without violating bars stock.\\
            Increment $cut_p$ frequency in $solution$ by $max(bars\_needed,n)$ times.\\
        }
        \While{number of bars of length $\mathcal{L}_{\gamma}$ needed was not reached }{
            $ove_p \leftarrow$ random overlapping pattern that generates a bar of length $\mathcal{L}_{\gamma}$ that has not yet been selected.\\
            bars\_needed $\leftarrow$ number of bars of length $\mathcal{L}_{\gamma}$ required.\\
            n $\leftarrow$ number of times $ove_p$ can be added to $solution$ without violating bars stock.\\
            Increment $ove_p$ frequency in $solution$ by $max(bars\_needed,n)$ times.\\
        }
    }
    Remove from $solution$ the genes associated to patterns that are not used\\
    \Return{solution}
 \caption{Generate pseudo-random solution}
 \label{gerar_ind}
\end{algorithm}

We call this method pseudorandom because we choose the patterns to add to the solution randomly, although each pattern frequency in the solution is computed in such a way as to respect stock and satisfy the demand. The time complexity of the Algorithm \ref{gerar_ind} is $\mathcal{O}(Pq_c + \Gamma(H+O))$. Generating the initial population consists of creating of a number of individuals with the use of Algorithm \ref{gerar_ind} and selecting the best of them based on their fitness value according to the required population size.

\subsection{Fitness function and selection operator}

We use the objective function \ref{FOCWP} from the mathematical model (ICP) as the fitness function to evaluate the solution quality of a given chromosome. The selection operator consists of the process of selecting the best distinct solutions with respect to their respective fitness function value, i.e., the individuals with the lowest fitness values. 

\subsection{Crossover operators}

In this subsection we propose two alternatives to use as crossover operators: crossover type 1, and crossover type 2. Given two parents, both crossover types generate one offspring, which consists of a new solution (chromosome). 

In crossover type 1, we preserve all pattern indices from both parents, but the number of times each pattern is used in the offspring corresponds to the mean of the number of times they are used by the parents rounded to the largest integer. For each gene there is a probability of mutation. When the mutation occurs the number of times that the current pattern is used in such gene is set to zero. After this crossover process, if the generated offspring results in an infeasible solution, an iterative procedure, shown in Algorithm \ref{fix_solution}, is applied for its correction. If some pattern from the current offspring is used zero times, the gene associated to it is removed from the chromosome.

In crossover type 2, we first initiate the offspring using all patterns that used in both parents with their respective frequencies set to zero. For the genes that have patterns that are part of both parents simultaneously, their respective frequencies are set as the mean of their frequencies in the parents rounded to the largest integer. For each remaining gene we have a probability of 50\% of setting its respective frequency to be equal to the originating parent frequency or keeping it equal to zero. If the resulting offspring is not feasible, the fixing procedure, shown in Algorithm \ref{fix_solution}, is applied to it and all patterns with final frequencies equal to zero have their respective genes removed from the chromosome.

\subsection{Mutation operator}

The mutation of an individual consists of choosing one pattern $p_{1}$ that is in the solution, and in the addition of one pattern $p_{2}$, chosen randomly, that is not part of the solution. The number of times that $p_{2}$ is used becomes the number of times that $p_1$ is used, and the number of times that $p_1$ is used is set to zero. If the solution is infeasible after this procedure we apply the fixing phase to it. This process is frequently required in practice and is described in this next subsection.

\subsection{Infeasible solution fixing}

Since that the proposed genetic operators of crossover and mutation can affect the feasibility of solutions, we must define a procedure to fix infeasible solutions to turn them into feasible ones before.

A chromosome may be an infeasible solution due to different reasons, as follows:

\begin{enumerate}
    \item Infeasibility type 1, due to beam demand: the frequencies of packing patterns in the solution are not enough to fulfill the beam demands;
    \item Infeasibility type 2, due to bar stock: the number of bars which are used in cutting and overlapping patterns exceed the bar stock;
    \item Infeasibility type 3, due to inconsistent number of bars produced and required: the number of necessary bars generated by cutting and overlapping patterns is different from the number of bars that beam production requires.
\end{enumerate}

If we detect any of those kinds of infeasibility, we must apply the infeasible solution fixing phase, which consists of Algorithm \ref{fix_solution}. Each infeasibility type is treated in a particular procedure: Algorithms \ref{fix_solution1}, \ref{fix_solution2}, and \ref{fix_solution3} are used to fix infeasibility type 1, 2, and 3, respectively. 

\begin{algorithm}[H]
 \KwData{Infeasible chromosome}
 \KwResult{Potentially modified chromosome}
    Initialize produced beams with zeros;\\
    demand\_fulfilled $\leftarrow$ false;\\
    \For{each packing pattern $P_i$ in Chromosome}{
        \uIf{demand\_fulfilled = false}{
            \For{cont = 1,\ldots, frequency($P_i$)}{
                Update produced beams;\\
                \If{produced beams fulfill the beam demands}
                {
                    demand\_fulfilled $\leftarrow$ true;\\
                    frequency($P_i$) $\leftarrow$ cont;\\
                    break;\\
                }
            }
        }\Else{
            frequency($P_i$) $\leftarrow$ 0\\
            }
    }
    \Return{Chromosome}
 \caption{Remove unnecessary packing patterns}
 \label{fix_solution0}
\end{algorithm}

\begin{algorithm}[H]
 \KwData{Infeasible chromosome}
 \KwResult{Potentially feasible chromosome}
    \While{Infeasibility type 1 = true}{
        \For{each beam type c}{
            \For{each beam length $l_c$ whose demand is not fulfilled}{
                \For{each packing pattern $P_i$ with type c in Chromosome}{
                    \If{frequency of $l_c$ in $P_i$ > 0}{
                        Increment frequency($P_i$) until the demand of $l_c$ is achieved;\\
                        break;\\
                    }
                }
            }
        }
    }
    \Return{Chromosome}
 \caption{Fix chromosome with respect to infeasibility 1}
 \label{fix_solution1}
\end{algorithm}

\begin{algorithm}[H]
 \KwData{Infeasible chromosome}
 \KwResult{Potentially feasible chromosome}
    Calculate the \#bars used;\\
    \For{each standard bar or bar leftover $w$}{
        \If{\#bars $w$ used > stock of $w$ bars}{
            \For{each cutting pattern $I_h$ that uses $w$ in Chromosome}{
                rt $\leftarrow$ \#bars $w$ used  - stock of $w$ bars; \\
                frequency($I_h$) $\leftarrow$ frequency($I_h$) - min(frequency($I_h$), rt);\\
                Update the \#bars $w$ used;\\
                \If{\#bars $w$ used > stock of $w$ bars}{
                    break;\\
                }
            }
            
            \For{each overlapping pattern $O_{\mu}$ that uses $w$ in Chromosome}{
                rt $\leftarrow$ \#bars $w$ used  - stock of $w$ bars; \\
                rt $\leftarrow$ $\floor*{\dfrac{rt}{\text{\#bars w in } O_{\mu}}}$\\
                frequency($O_{\mu}$) $\leftarrow$ frequency($O_{\mu}$) - min(frequency($O_{\mu}$),rt);\\
                Update the \#bars $w$ used;\\
                \If{\#bars $w$ used > stock of $w$ bars}{
                    break;\\
                }
            }
        }
    }
    \Return{Chromosome}
 \caption{Fix chromosome with respect to infeasibility 2}
 \label{fix_solution2}
\end{algorithm}

\begin{algorithm}[H]
 \KwData{Infeasible chromosome}
 \KwResult{Potentially feasible chromosome}
    Calculate the \#bars generated by cutting and overlapping patterns;\\
    Calculate the \#bars that beam production requires according to the frequency of packing patterns;\\
    \For{each bar $\gamma$ generated}{
        \If{\#bars $\gamma$ generated > \#bars $\gamma$ that beam production requires}{
            \For{each cutting pattern $I_h$ that generates only bars $\gamma$}{
                rt $\leftarrow$ \#bars $\gamma$ generated  - \#bars $\gamma$ that beam production requires; \\
                rt $\leftarrow$ $\ceil*{\dfrac{rt}{\text{\#bars $\gamma$ generated by } I_{h}}}$\\
                frequency($I_h$) $\leftarrow$ frequency($I_h$) - min(frequency($I_h$), rt);\\
                Update the \#bars $\gamma$ generated;\\
            }
        } 
        
        \If{\#bars $\gamma$ generated > \#bars $\gamma$ that beam production requires}{
            \For{each overlapping pattern $O_{\mu}$ that generates a bar $\gamma$}{
                rt $\leftarrow$ \#bars $\gamma$ generated  - \#bars $\gamma$ that beam production requires; \\
                frequency($O_{\mu}$) $\leftarrow$ frequency($O_{\mu}$) - min(frequency($O_{\mu}$), rt);\\
                Update the \#bars $\gamma$ generated;\\
            }
        }

        \If{\#bars $\gamma$ generated < \#bars $\gamma$ that beam production requires}{
            \For{each cutting pattern $I_h$ that generates only bars $\gamma$}{
                rt $\leftarrow$ \#bars $\gamma$ that beam production requires - number bars $\gamma$ generated; \\
                rt $\leftarrow$ $\floor*{\dfrac{rt}{\text{\#bars $\gamma$ generated by } I_{h}}}$\\
                frequency($I_h$) frequency $\leftarrow$ frequency($I_h$) + min(rt,stock of $\gamma$ bars remaining) ;\\
                Update the \#bars $\gamma$ generated;\\
            }
        } 
        
        \If{\#bars $\gamma$ generated < \#bars $\gamma$ that beam production requires}{
            \For{each overlapping pattern $O_{\mu}$ that generates a bar $\gamma$}{
                Increment frequency($O_{\mu}$) until (\#bars $\gamma$ generated $\geq$ \#bars $\gamma$ that beam production requires) or the stock is violated with new increment;\\
                Update the \#bars $\gamma$ generated;\\
            }
        }
    }
    \Return{Chromosome}
 \caption{Fix chromosome with respect to infeasibility 3}
 \label{fix_solution3}
\end{algorithm}

\begin{algorithm}[H]
 \KwData{Infeasible chromosome}
 \KwResult{Possible feasible chromosome}
    \uIf{Infeasibility type 1 = true}{
        Call Algorithm \ref{fix_solution1};\\ 
    }\Else{
        Call Algorithm \ref{fix_solution0};\\}
    \If{Infeasibility type 2 = true}{
        Call Algorithm \ref{fix_solution2};\\ 
    }
    \If{Infeasibility type 3 = true}{
        Call Algorithm \ref{fix_solution3};\\ 
    }
    \Return{chromosome}
 \caption{Solution fixing procedure}
 \label{fix_solution}
\end{algorithm}

The unnecessary packing patterns procedure, shown in Algorithm \ref{fix_solution0}, in Appendix A, works like a solution treatment phase, which is not a necessary part of the solution fixing process, although applying such procedure we may improve solution quality and simplify the fixing process, i.e., it would be less likely that the modified solutions could not be fixed. The procedure consists of decreasing the frequency of packing patterns after the beam demands are already fulfilled if there are beam surplus.

In Figure \ref{fig:crossover}, we show an example of the crossover operators, with offspring 1 as the solution generated by crossover operator type 1, and offspring 2 as the solution created by crossover operator type 2. Note that the fixing procedure was applied for offspring 2 and not for offspring 1. In Figure \ref{fig:mutation}, we show an example of the proposed mutation operator. The resulting chromosome is infeasible, therefore, the solution fixing procedure must be applied. If the application of the solution fixing procedure to a given chromosome could not turn it into a feasible solution, the chromosome is discarded.

\begin{figure}[H]
    \centering
    \caption{Crossover operators}
    \includegraphics[width=.45\linewidth]{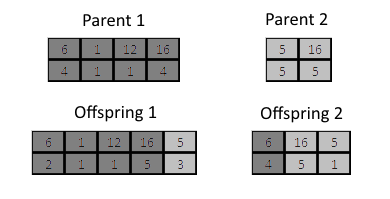}
    \label{fig:crossover}
\end{figure}

\begin{figure}[H]
    \centering
    \caption{Mutation operator and solution correction}
    \includegraphics[width=.45\linewidth]{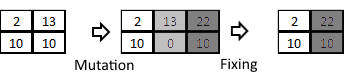}
    \label{fig:mutation}
\end{figure}

\subsection{Population restart}

The population restart consists of the creation of a new population to compose the next generation after a predefined number of epochs.  We apply a population restart after a given number of generations with no improvement of the best-fitness value. We divide such procedure into three parts, as follows: 1. selecting a certain number of the best-fitness individuals from the current population; 2. generating a number new pseudo-random individuals; 3. creating a new population with individuals from steps 1 and 2 and applying the selection operator to form the next population.

\subsection{Local search}

In order to improve the quality of final solutions, we apply a local search to every individual of the final population. For the local search we use the \textit{insert} movement, which consists of, given two genes indices $i$ and $k$, with $i < k$, inserting the gene $i$ one position in front of $k$-th gene, i.e., all the genes between positions $i$ and $k + 1$ are moved one position to the right after the insertion of the $k$-th gene. In Figure \ref{fig:localsearch} an insert movement neighbor is shown for a given solution after inserting 2nd gene in front of 5th gene.

\begin{figure}[H]
    \centering
    \caption{Insert movement}
    \includegraphics[width=.25\linewidth]{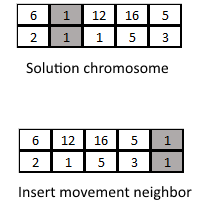}
    \label{fig:localsearch}
\end{figure}

Considering the function INSERT(solution, i, k) as the movement of insertion given indices $i$ and $k$, we describe the local search procedure in the Algorithm \ref{insertN}.

\begin{algorithm}[H]
 \KwData{\textit{InitialSolution}}
 \KwResult{\textit{BestSolution}}
    {$BestSolution \leftarrow InitialSolution;$} \\
    \For{$i = 1, \ldots, n\ell - 1$}{
        \For{$k = k + 1, \ldots, n\ell$}{
            {$neighbor\leftarrow \text{INSERT}(InitialSolution, i, k)$;}\\
            \If{$makespan(neighbor) < makespan(BestSolution)$}{
            $BestSolution \leftarrow neighbor$;
            }
        }
    }
    \Return{BestSolution;}
 \caption{Insert neighborhood}
 \label{insertN}
\end{algorithm}

\subsection{Algorithm description}

In order to describe the proposed genetic algorithm we define the following parameters: population size (TP), number of generations (NG), crossover type (CRS), number of pseudo-random solutions generated for the initial population and restart selections (AS), mutation probability (MUT), number of generations with no fitness improvement to apply population restart (RST), and the number of individuals from the current population selected to be used in restart operator procedure (TER).

The proposed genetic algorithm can be seen as a steady-state model since only one new individual is generated per generation, even though we generate several individuals in the formation of the initial population and in a population restart process. A simplified scheme of the proposed genetic algorithm is shown in the flowchart in Figure \ref{fig:GA_flow}.

\begin{figure}[H]
    \centering
    \caption{Simplified flowchart of proposed genetic algorithm }
    \includegraphics[width = .7\linewidth]{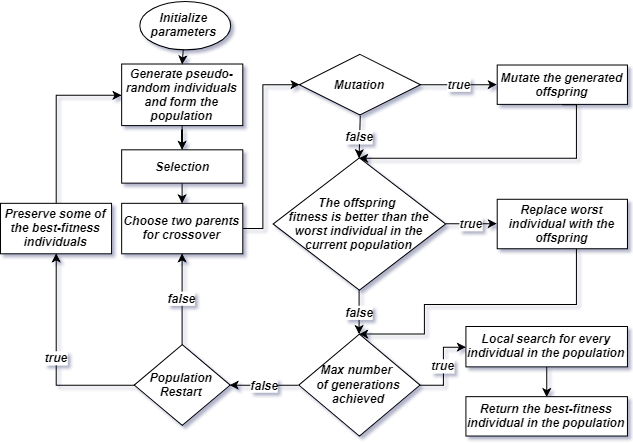}
    \label{fig:GA_flow}
\end{figure}

\section{Computational experiments}
\label{P2:tests}

In this section we present computational experiments on a set of benchmark instances that were produced with the intent to mimic real-world scenarios, to evaluate the solution methods proposed in this study.

The patterns corresponding to each test instance were generated using the constraint programming solver IBM ILOG CPLEX 12.8 CP Optimizer. For the integer programming model implementation we adopted the solver IBM ILOG CPLEX 12.8. Both solvers were used with Concert technology using the C++ programming language. The genetic algorithms were also developed with the C++ programming language.

We carried out every test in this paper on a Linux Ubuntu 18.04 64bits machine with 8GB of memory and Intel Core i5-3470 CPU 3.20 GHz $\times 4$ processor. We compiled the created codes with the GNU GCC 7.3.0 compiler using Code::Blocks 17.12 IDE. Note that, for different values of $\lambda_i$ we can form the Pareto front and may have different behaviors of the proposed model and algorithms. However, for the purpose of the study, we did not approach the multi-objective nature of the problem and considered, for each test described in this section, $\lambda_i = 1$, with $i = 1, \ldots, 4$.

\subsection{Test instances generation}
\label{T_instance_g}

In this subsection, we describe how we generate the set of benchmark instances used in this section. We introduce a set of instances that are based on data arising from a possible real-world scenario. The different instances represent a sample of the variability of the problem's parameters, such as number of beam types, number of molds, and mold lengths.

In Table \ref{P2_instances_description} we present details about each test instance parameter. We can see that the number of packing patterns increases as the number of beam types increases. However, the number of cutting and overlapping patterns remains constant because of the fact that we expect that the possible distinct bar lengths are standardized in real-world scenarios and therefore do not lead to variability.

\begin{table}[H]
\caption{Description of test instances}
\centering
\begin{tabular}{ccccccccccccccc}
\hline
\textbf{Instance} & {$C$} & {$M$} & {$T$} & {$r$} & {$H$} & {$O$} & \textit{\textbf{}} & \textbf{Instance} & {$C$} & {$M$} & {$T$} & {$r$} & {$H$} & {$O$} \\ \cline{1-7} \cline{9-15} 
cwp001 & 1 & 15 & 6 & 145 & 10 & 12 &  & cwp036 & 4 & 30 & 20 & 715 & 10 & 12 \\
cwp002 & 1 & 15 & 6 & 199 & 10 & 12 &  & cwp037 & 4 & 30 & 24 & 679 & 10 & 12 \\
cwp003 & 1 & 15 & 6 & 236 & 10 & 12 &  & cwp038 & 4 & 30 & 15 & 702 & 10 & 12 \\
cwp004 & 1 & 15 & 6 & 210 & 10 & 12 &  & cwp039 & 4 & 30 & 14 & 732 & 10 & 12 \\
cwp005 & 1 & 15 & 6 & 236 & 10 & 12 &  & cwp040 & 4 & 30 & 30 & 750 & 10 & 12 \\
cwp006 & 1 & 30 & 3 & 257 & 10 & 12 &  & cwp041 & 5 & 15 & 68 & 966 & 10 & 12 \\
cwp007 & 1 & 30 & 3 & 257 & 10 & 12 &  & cwp042 & 5 & 15 & 57 & 927 & 10 & 12 \\
cwp008 & 1 & 30 & 3 & 199 & 10 & 12 &  & cwp043 & 5 & 15 & 66 & 985 & 10 & 12 \\
cwp009 & 1 & 30 & 3 & 218 & 10 & 12 &  & cwp044 & 5 & 15 & 59 & 983 & 10 & 12 \\
cwp010 & 1 & 30 & 3 & 199 & 10 & 12 &  & cwp045 & 5 & 15 & 75 & 1046 & 10 & 12 \\
cwp011 & 2 & 15 & 15 & 414 & 10 & 12 &  & cwp046 & 5 & 30 & 29 & 974 & 10 & 12 \\
cwp012 & 2 & 15 & 21 & 395 & 10 & 12 &  & cwp047 & 5 & 30 & 29 & 926 & 10 & 12 \\
cwp013 & 2 & 15 & 21 & 361 & 10 & 12 &  & cwp048 & 5 & 30 & 24 & 949 & 10 & 12 \\
cwp014 & 2 & 15 & 14 & 387 & 10 & 12 &  & cwp049 & 5 & 30 & 30 & 1008 & 10 & 12 \\
cwp015 & 2 & 15 & 17 & 451 & 10 & 12 &  & cwp050 & 5 & 30 & 27 & 1062 & 10 & 12 \\
cwp016 & 2 & 30 & 8 & 466 & 10 & 12 &  & cwp051 & 6 & 15 & 62 & 1249 & 10 & 12 \\
cwp017 & 2 & 30 & 8 & 352 & 10 & 12 &  & cwp052 & 6 & 15 & 51 & 1204 & 10 & 12 \\
cwp018 & 2 & 30 & 9 & 459 & 10 & 12 &  & cwp053 & 6 & 15 & 51 & 1221 & 10 & 12 \\
cwp019 & 2 & 30 & 8 & 500 & 10 & 12 &  & cwp054 & 6 & 15 & 62 & 1291 & 10 & 12 \\
cwp020 & 2 & 30 & 9 & 466 & 10 & 12 &  & cwp055 & 6 & 15 & 65 & 1371 & 10 & 12 \\
cwp021 & 3 & 15 & 29 & 662 & 10 & 12 &  & cwp056 & 6 & 30 & 21 & 1324 & 10 & 12 \\
cwp022 & 3 & 15 & 36 & 643 & 10 & 12 &  & cwp057 & 6 & 30 & 33 & 1279 & 10 & 12 \\
cwp023 & 3 & 15 & 30 & 614 & 10 & 12 &  & cwp058 & 6 & 30 & 33 & 1305 & 10 & 12 \\
cwp024 & 3 & 15 & 29 & 671 & 10 & 12 &  & cwp059 & 6 & 30 & 35 & 1052 & 10 & 12 \\
cwp025 & 3 & 15 & 35 & 684 & 10 & 12 &  & cwp060 & 6 & 30 & 32 & 1165 & 10 & 12 \\
cwp026 & 3 & 30 & 15 & 589 & 10 & 12 &  & cwp061 & 7 & 15 & 60 & 1427 & 10 & 12 \\
cwp027 & 3 & 30 & 18 & 560 & 10 & 12 &  & cwp062 & 7 & 15 & 86 & 1396 & 10 & 12 \\
cwp028 & 3 & 30 & 18 & 433 & 10 & 12 &  & cwp063 & 7 & 15 & 113 & 1211 & 10 & 12 \\
cwp029 & 3 & 30 & 17 & 620 & 10 & 12 &  & cwp064 & 7 & 15 & 53 & 1438 & 10 & 12 \\
cwp030 & 3 & 30 & 20 & 557 & 10 & 12 &  & cwp065 & 7 & 15 & 89 & 1395 & 10 & 12 \\
cwp031 & 4 & 15 & 45 & 952 & 10 & 12 &  & cwp066 & 7 & 30 & 36 & 1243 & 10 & 12 \\
cwp032 & 4 & 15 & 50 & 650 & 10 & 12 &  & cwp067 & 7 & 30 & 45 & 1568 & 10 & 12 \\
cwp033 & 4 & 15 & 45 & 896 & 10 & 12 &  & cwp068 & 7 & 30 & 38 & 1403 & 10 & 12 \\
cwp034 & 4 & 15 & 41 & 839 & 10 & 12 &  & cwp069 & 7 & 30 & 39 & 1487 & 10 & 12 \\
cwp035 & 4 & 15 & 41 & 783 & 10 & 12 &  & cwp070 & 7 & 30 & 39 & 1494 & 10 & 12 \\ \hline
\end{tabular}
\label{P2_instances_description}
\end{table}

We consider mold capacities of 5.95m and 11.95m, while we take 1.12m, 1.45m, 2.35m, 2.5m, 2.65m, 2.95m, and 3.3m as possible beam lengths. For each instance, the possible curing times may be 1, 2, or 3 periods, chosen randomly when instances have more than 3 types. In addition, if the instance has up to 3 beam types, we associate the curing time to the beam type index, for example the beam type 2 needs a curing time of 2 periods. With respect to the number of bars that some beam type demands, we choose randomly a value between 1 and 3 for each beam type. We choose the beam demands uniformly between 17 and 50. For total time horizon $T$, we calculate it as the ceiling of 150\% of the optimal makespan lower bound, defined by Equation \ref{calculoT_p2} as follows:

\begin{equation}
    T = \ceil*{1.5\cdot{\displaystyle\sum_{i = 1}^C t_c \cdot \left( \sum_{k=1}^{q_c} l(c,k) \cdot d(c,k) \right)}/{\displaystyle\sum_{m =1}^M L_m}}.\label{calculoT_p2}
\end{equation}

For all instances, we consider an unique length of new bars as 12m and the possible lengths of bar leftovers as 2m, 5m, 6m, and 8m. We do not vary such lengths along the test instances since, in practice, it is expected that they are standardized. To generate realistic bar stocks we introduce an upper bound for the number of bars needed to fulfill the beam demand as $UB$, defined in Equation (\ref{UB_2}):

\begin{equation}
    UB = 2\cdot T\cdot M \cdot \max_{D_c = 1, \ldots, C}\{D_c\}. \label{UB_2}
\end{equation}

We set the stock of new bars of length 12m equal to $UB$, whilst we choose the stock of each leftover randomly between $\ceil{UB/5}$ and $UB$ following an uniform distribution. We implemented the instance generator using MATLAB programming language.

\subsection{Computational experiments with the mathematical model}

In this subsection we discuss the results of the computational tests with the benchmark instance set that we generated following the scheme described in Subsection \ref{T_instance_g}. In Table \ref{tabelona_p2} we show the results of the computational experiments for the model (ICP) and its linear relaxation. The solution time was limited to 3,600 seconds. As regards to notation in Table \ref{tabelona_p2}, we consider LB, IP, and LP standing for the optimal objective function value lower bound, best solution value by CPLEX for model (ICP), and its linear relaxation value, respectively. When we say \textit{gap} we mean the relative percentage deviation between the best integer objective and the objective of the best node remaining in the CPLEX $B\&C$ tree, calculated like this: $gap = 100\cdot|bestbound-bestinteger|/(1e-10+|bestinteger|)$ (0\% means a proven optimal solution). We denote by ``B\&C nodes'' the number of nodes generated in the branch-and-cut tree in the solution process, and \textit{t (s)} as the solution time in seconds.

We can see in Table \ref{tabelona_p2} that the linear relaxation of all instances could be solved, with the average time of 53.21 seconds, and with 624.61 seconds being the longest time to get to the optimal solution. On the other hand, only 11 instances could be solved to optimality by the integer programming model (4 of them solved in the root node of the B\&C tree). For 23 instances we could not even find a feasible solution, a situation that we denote by ``--''. Moreover, we could not solve 36 instances to optimality within the time limit, although feasible solutions for them were found. We can infer from the computational test results that the larger the instance parameter values are, the larger the problem is and the most difficult it is to find solutions for it. With high values of the instance parameters, when solutions are found, the optimality gap tends to be worse, i.e. the solutions achieved within the time limit are even further from the optimal solution.

\begin{table}[H]
\caption{Results of integer programming model and its linear relaxation}
\centering
\resizebox{1\linewidth}{!}{
\begin{tabular}{crrrrrlrrlcrrrrrlrr}
\hline
\textbf{} & \multicolumn{1}{c}{\textbf{}} & \multicolumn{4}{c}{\textbf{Mathematical Model}} & \multicolumn{1}{c}{\textbf{}} & \multicolumn{2}{c}{\textbf{Linear Relaxation}} & \multicolumn{1}{c}{\textbf{}} & \textbf{} & \multicolumn{1}{c}{\textbf{}} & \multicolumn{4}{c}{\textbf{Mathematical Model}} & \multicolumn{1}{c}{\textbf{}} & \multicolumn{2}{c}{\textbf{Linear Relaxation}} \\ \cline{3-6} \cline{8-9} \cline{13-16} \cline{18-19} 
\textbf{Instance} & \multicolumn{1}{c}{\textbf{LB}} & \multicolumn{1}{c}{\textbf{IP}} & \multicolumn{1}{c}{\textbf{B\&C nodes}} & \multicolumn{1}{c}{\textbf{gap}} & \multicolumn{1}{c}{\textbf{t (s)}} & \multicolumn{1}{c}{\textbf{}} & \multicolumn{1}{c}{\textbf{LP}} & \multicolumn{1}{c}{\textbf{t (s)}} & \multicolumn{1}{c}{\textbf{}} & \textbf{Instance} & \multicolumn{1}{c}{\textbf{LB}} & \multicolumn{1}{c}{\textbf{IP}} & \multicolumn{1}{c}{\textbf{B\&C nodes}} & \multicolumn{1}{c}{\textbf{gap}} & \multicolumn{1}{c}{\textbf{t (s)}} & \multicolumn{1}{c}{\textbf{}} & \multicolumn{1}{c}{\textbf{LP}} & \multicolumn{1}{c}{\textbf{t (s)}} \\ \cline{1-6} \cline{8-16} \cline{18-19} 
cwp001 & 5.55 & 6.05 & 981 & 0.00\% & 1.2 &  & 3.58 & 0.02 &  & cwp036 & 22.95 & 25.80 & 1,394 & 12.02\% & 3600.0 &  & 15.60 & 8.39 \\
cwp002 & 7.60 & 8.10 & 0 & 0.00\% & 1.2 &  & 6.02 & 0.04 &  & cwp037 & 33.60 & -- & -- & -- & 3600.0 &  & 22.94 & 11.47 \\
cwp003 & 9.25 & 9.70 & 189 & 0.00\% & 1.8 &  & 7.57 & 0.08 &  & cwp038 & 24.30 & 26.30 & 55,260 & 0.37\% & 3600.0 &  & 19.86 & 2.01 \\
cwp004 & 7.50 & 8.20 & 686 & 0.00\% & 2.3 &  & 5.80 & 0.04 &  & cwp039 & 29.05 & 32.00 & 79,965 & 1.32\% & 3600.0 &  & 25.56 & 3.32 \\
cwp005 & 8.45 & 9.70 & 75 & 0.00\% & 2.1 &  & 7.38 & 0.09 &  & cwp040 & 34.65 & -- & -- & -- & 3600.0 &  & 20.43 & 30.10 \\
cwp006 & 8.15 & 8.15 & 0 & 0.00\% & 1.1 &  & 7.51 & 0.07 &  & cwp041 & 66.15 & -- & -- & -- & 3600.0 &  & 35.96 & 148.02 \\
cwp007 & 3.70 & 4.15 & 0 & 0.00\% & 1.3 &  & 2.76 & 0.04 &  & cwp042 & 55.00 & -- & -- & -- & 3600.0 &  & 29.79 & 51.94 \\
cwp008 & 5.40 & 5.50 & 0 & 0.00\% & 0.7 &  & 4.55 & 0.04 &  & cwp043 & 67.00 & -- & -- & -- & 3600.0 &  & 37.59 & 99.51 \\
cwp009 & 6.15 & 7.20 & 3,640,484 & 0.69\% & 3600.0 &  & 5.54 & 0.08 &  & cwp044 & 57.75 & -- & -- & -- & 3600.0 &  & 31.42 & 137.86 \\
cwp010 & 6.35 & 7.00 & 2,455 & 0.00\% & 3.1 &  & 5.83 & 0.05 &  & cwp045 & 69.25 & -- & -- & -- & 3600.0 &  & 32.80 & 118.61 \\
cwp011 & 17.30 & 19.70 & 221,350 & 0.62\% & 3600.0 &  & 12.31 & 0.66 &  & cwp046 & 41.00 & -- & -- & -- & 3600.0 &  & 28.89 & 26.76 \\
cwp012 & 22.85 & 26.85 & 451,142 & 0.32\% & 3600.0 &  & 14.74 & 0.64 &  & cwp047 & 31.65 & 35.90 & 3,240 & 2.09\% & 3600.0 &  & 19.90 & 19.33 \\
cwp013 & 23.10 & 25.50 & 564,930 & 0.32\% & 3600.0 &  & 15.43 & 0.75 &  & cwp048 & 33.80 & 38.15 & 1,112 & 11.12\% & 3600.0 &  & 24.39 & 29.49 \\
cwp014 & 13.95 & 14.90 & 754,750 & 0.39\% & 3600.0 &  & 9.75 & 0.88 &  & cwp049 & 37.55 & 47.05 & 2 & 13.61\% & 3600.0 &  & 24.11 & 11.40 \\
cwp015 & 18.75 & 21.00 & 561,739 & 0.49\% & 3600.0 &  & 12.85 & 0.89 &  & cwp050 & 38.40 & -- & -- & -- & 3600.0 &  & 26.84 & 57.54 \\
cwp016 & 8.45 & 8.90 & 2,422 & 0.00\% & 49.8 &  & 5.84 & 1.19 &  & cwp051 & 61.70 & 77.20 & 0 & 17.09\% & 3600.0 &  & 35.92 & 46.80 \\
cwp017 & 14.05 & 15.70 & 2,711 & 0.00\% & 33.6 &  & 11.55 & 0.60 &  & cwp052 & 52.65 & -- & -- & -- & 3600.0 &  & 35.71 & 56.03 \\
cwp018 & 15.05 & 18.80 & 775,314 & 1.34\% & 3600.0 &  & 11.53 & 0.59 &  & cwp053 & 51.80 & 57.30 & 1379 & 2.87\% & 3600.0 &  & 33.99 & 20.94 \\
cwp019 & 13.60 & 17.50 & 457,947 & 0.59\% & 3600.0 &  & 11.04 & 0.54 &  & cwp054 & 60.35 & -- & -- & -- & 3600.0 &  & 35.63 & 139.89 \\
cwp020 & 14.95 & 17.40 & 406,983 & 1.02\% & 3600.0 &  & 11.40 & 0.82 &  & cwp055 & 64.90 & -- & -- & -- & 3600.0 &  & 36.53 & 45.26 \\
cwp021 & 30.00 & 33.00 & 3,554 & 12.29\% & 3600.0 &  & 18.53 & 5.71 &  & cwp056 & 35.20 & 39.15 & 1289 & 4.42\% & 3600.0 &  & 28.82 & 23.40 \\
cwp022 & 36.40 & 39.85 & 6,188 & 0.98\% & 3600.0 &  & 20.81 & 5.37 &  & cwp057 & 43.15 & -- & -- & -- & 3600.0 &  & 28.48 & 135.68 \\
cwp023 & 31.50 & 35.20 & 6,342 & 5.39\% & 3600.0 &  & 19.36 & 7.23 &  & cwp058 & 48.30 & -- & -- & -- & 3600.0 &  & 34.29 & 71.79 \\
cwp024 & 25.70 & 27.45 & 12,286 & 0.38\% & 3600.0 &  & 13.33 & 3.50 &  & cwp059 & 51.20 & -- & -- & -- & 3600.0 &  & 36.29 & 54.39 \\
cwp025 & 34.65 & 37.40 & 7,495 & 6.44\% & 3600.0 &  & 19.82 & 5.04 &  & cwp060 & 43.30 & 62.35 & 0 & 28.21\% & 3600.0 &  & 29.78 & 40.41 \\
cwp026 & 21.20 & 22.40 & 12,095 & 2.11\% & 3600.0 &  & 15.16 & 5.31 &  & cwp061 & 69.15 & 79.80 & 21 & 27.35\% & 3600.0 &  & 46.85 & 336.39 \\
cwp027 & 22.60 & 24.15 & 7,226 & 5.92\% & 3600.0 &  & 15.03 & 5.03 &  & cwp062 & 76.40 & -- & -- & -- & 3600.0 &  & 40.02 & 469.50 \\
cwp028 & 25.40 & 26.30 & 9,273 & 8.06\% & 3600.0 &  & 18.15 & 6.43 &  & cwp063 & 108.55 & -- & -- & -- & 3600.0 &  & 53.72 & 103.78 \\
cwp029 & 23.50 & 25.80 & 7,529 & 5.25\% & 3600.0 &  & 16.67 & 6.79 &  & cwp064 & 64.85 & 73.20 & 103 & 2.84\% & 3600.0 &  & 48.86 & 25.38 \\
cwp030 & 26.90 & 30.00 & 8,024 & 1.09\% & 3600.0 &  & 18.56 & 4.12 &  & cwp065 & 86.60 & -- & -- & -- & 3600.0 &  & 45.36 & 624.61 \\
cwp031 & 42.80 & 47.00 & 1,530 & 9.04\% & 3600.0 &  & 22.88 & 14.88 &  & cwp066 & 48.45 & -- & -- & -- & 3600.0 &  & 33.76 & 187.76 \\
cwp032 & 50.40 & 58.10 & 365 & 14.42\% & 3600.0 &  & 28.21 & 23.74 &  & cwp067 & 60.95 & -- & -- & -- & 3600.0 &  & 40.96 & 126.30 \\
cwp033 & 44.20 & 50.10 & 3,600 & 12.07\% & 3600.0 &  & 24.26 & 42.86 &  & cwp068 & 62.65 & 82.55 & 204 & 35.26\% & 3600.0 &  & 47.89 & 61.85 \\
cwp034 & 40.70 & 44.10 & 2,102 & 3.91\% & 3600.0 &  & 24.56 & 6.51 &  & cwp069 & 51.05 & -- & -- & -- & 3600.0 &  & 34.18 & 97.62 \\
cwp035 & 43.35 & -- & -- & -- & 3600.0 &  & 27.33 & 12.67 &  & cwp070 & 62.50 & -- & -- & -- & 3600.0 &  & 46.39 & 137.69 \\ \hline
\end{tabular}
}
\label{tabelona_p2}
\end{table}

We compare the results of the integer linear model (ICP), its linear relaxation, and our lower bound, in Equation \ref{CPlowerbound}, for the optimal value of objective function in the chart in Figure \ref{FigureLB}. 

\begin{figure}[!ht]
    \centering
    \caption{Objective function values for integer model solutions, linear relaxation solutions and proposed lower bound value for test instances}
    \includegraphics[width = 0.6\linewidth]{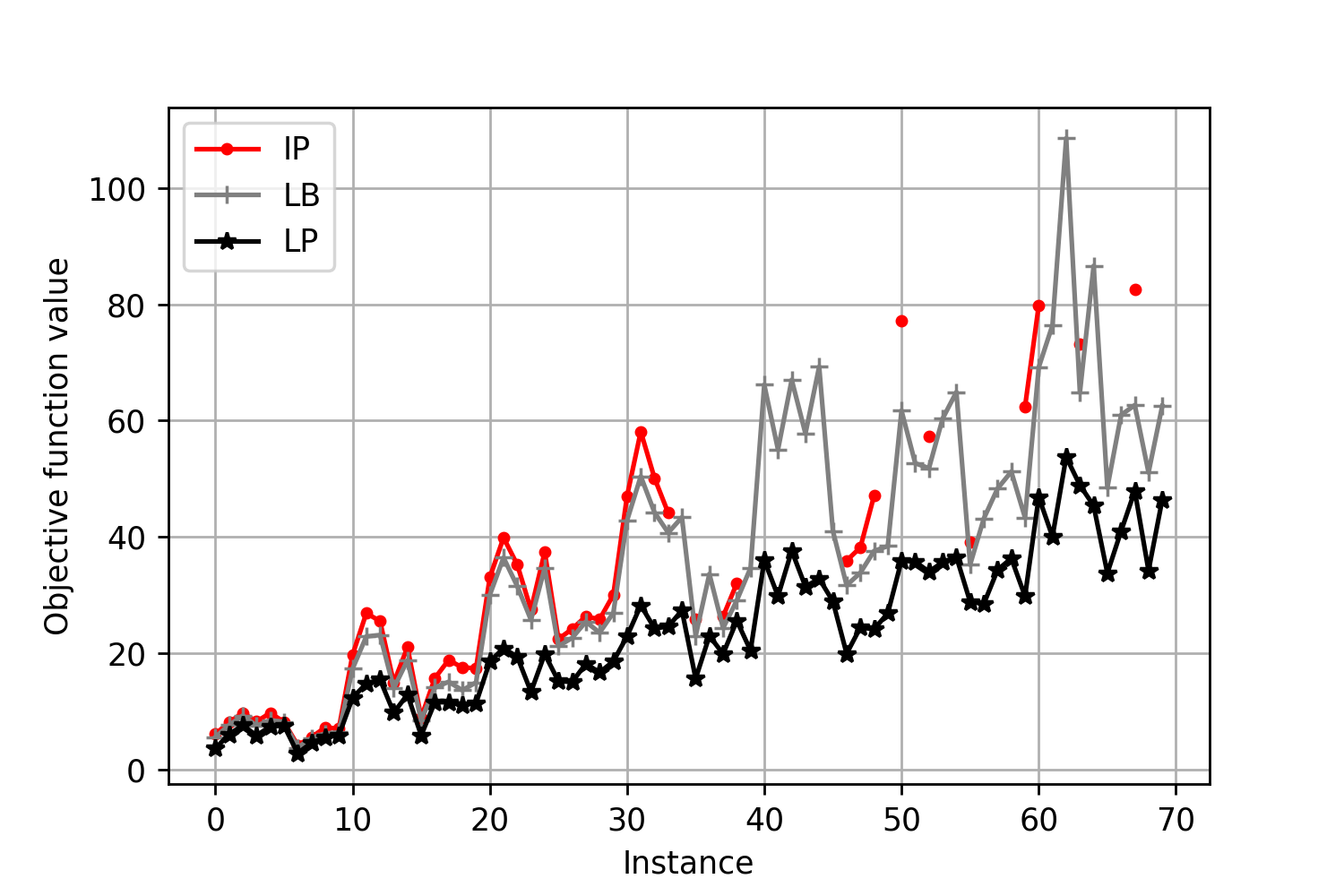}
    \label{FigureLB}
\end{figure}

In Figure \ref{FigureLB}, the lower bound proposed in this work for the optimal objective function value was greater than the linear relaxation for all test instances and highly close to the objective function values obtained by CPLEX.

\subsection{Experimental design and computational experiments with the proposed genetic algorithm}

In order to achieve a better parameterization for the robustness of the proposed genetic algorithm, we apply fractional factorial parameter design. \cite{gholami2009scheduling} used Taguchi experimental design \citep{pignatiello1988overview} to achieve improved robustness of the genetic algorithm which they proposed. In this method the optimal parameter choice is found with the analysis of different level combinations of the control factors in an orthogonal array, with no necessity of testing all of the possible level combinations. \ref{factorlevels} displays the proposed levels for the genetic algorithm parameters (control factors) introduced in Section \ref{P2:GA}.

\begin{table}[H]
\centering
\caption{Factor levels}
\begin{tabular}{llcll}
\hline
\textbf{Factors} &  & \textbf{Index of levels} &  & \textbf{Levels} \\ \hline
$TP$ &  & 1 &  & 25 \\
 &  & 2 &  & 50 \\\hline
$NG$ &  & 1 &  & 500$\cdot r$ \\
 &  & 2 &  & 1000$\cdot r$ \\\hline
$MUT$ &  & 1 &  & 0.01 \\
 &  & 2 &  & 0.025 \\
 &  & 3 &  & 0.05 \\\hline
$RST$ &  & 1 &  & $\ceil{ 0.1 \cdot NG }$\\
 &  & 2 &  & $\ceil{0.2 \cdot NG} $\\\hline
$AS$ &  & 1 &  & 100$\cdot r$ \\
 &  & 2 &  & 500$\cdot r$ \\\hline
$CRS$ &  & 1 &  & Type 1 \\
 &  & 2 &  & Type 2 \\\hline
$TER$ &  & 1 &  & $\ceil{ 0.1 \cdot T r}$ \\
 &  & 2 &  & $\ceil{0.2 \cdot T r}$ \\ \hline
\end{tabular}
\label{factorlevels}
\end{table}

We must have one degree of freedom for total mean, one degree of freedom for each factor with two levels, and two degrees of freedom for the factor with 3 levels amounting to a total of nine degrees of freedom ($1+1\times6+2\times1 = 9$). However, with the control factors and respective levels that we defined, there is no orthogonal array aside from the full factorial array. Thus, we are not able to use a classical Taguchi orthogonal array design. In such circumstances one alternative is to use the \textit{D-optimal} design\citep{de1995d}, which are constructed to minimize the generalized variance of the estimated regression coefficients. Note that D-optimality is only one possible criterion to choose a particular design. We obtain the \textit{D-optimal} design, by Fedorov algorithm \citep{triefenbach2008design} using R programming language for 9 trials for the chosen factors and their respective levels, illustrated in Table \ref{ortho_design}.

\begin{table}[H]
\centering
\caption{\textit{D-optimal} design with 9 trials}
\begin{tabular}{cccccccc}
\hline
\textbf{Trial} & \textbf{\textit{TP}} & \textbf{\textit{NG}} & \textbf{\textit{MUT}} & \textbf{\textit{RST}} & \textbf{\textit{AS}} & \textbf{\textit{CRS}} & \textbf{\textit{TER}} \\ \hline
1 & 1 & 1 & 1 & 2 & 2 & 1 & 1 \\
2 & 1 & 2 & 3 & 1 & 1 & 2 & 1 \\
3 & 1 & 2 & 2 & 1 & 2 & 1 & 2 \\
4 & 1 & 1 & 2 & 2 & 1 & 2 & 2 \\
5 & 2 & 2 & 2 & 2 & 1 & 1 & 1 \\
6 & 2 & 1 & 2 & 1 & 2 & 2 & 1 \\
7 & 2 & 1 & 3 & 1 & 1 & 1 & 2 \\
8 & 2 & 2 & 1 & 1 & 1 & 2 & 2 \\
9 & 2 & 2 & 3 & 2 & 2 & 2 & 2 \\ \hline
\end{tabular}
\label{ortho_design}
\end{table}

Furthermore, the effectiveness characteristic of the genetic algorithms proposed is the expected fitness value, which we seek to minimize, i.e., ``the lower is better principle''. Thus, for increased robustness of the algorithm we use the S/N (signal-to-noise) ratio, defined as follows. Note that the larger the value of S/N ratio the better.

\begin{equation}
    \text{S/N ratio:} \quad \eta_i = - 10 \ln \Bigg(\dfrac{1}{N}\sum_{j = 1}^N FIT_{ij}^2\Bigg),
\end{equation}

\noindent with $i$ and $j$ denoting index of trial and index of replication, while $FIT$ stands for the best objective function value obtained by running the GA. We denote by ``trial'' a certain combination of the control factor levels.

We define a replication as one GA run of some trial for a given instance, and $N$ as the number of test instances multiplied by the number of replications. Since we have an instance set of size 70 and we run each instance 10 times, we perform 700 replications for each trial.

Since CPLEX could not find optimal or even a feasible solution for most instances, we are unable to use the relative percentage deviations from the optimal solution as a performance indicator for the GA. Thus, we use the lower bound relative percentage deviations (LBD) of the fitness values for such purpose. Given a trial $i$ and a replication $j$ the LBD value is defined as follows:

\begin{equation}
    LBD_{ij} = \dfrac{FIT_{ij} - \text{LB}_{j}}{\text{LB}_{j}},
\end{equation}

\noindent where $\text{LB}_{j}$ stands for the lower bound of the optimal objective function value for the test instance used in replication $j$. The $LBD$ for a given trial $i$, denoted by $LBD_i$, is the average $LBD$ for all replications of instance set, as we can see in the following equation:

\begin{equation}
    LBD_{i} = \frac{1}{N} \cdot \sum_{j = 1}^N LBD_{ij},
\end{equation}

The remainder of the experimental design procedure consists of three phases: 

\begin{enumerate}
    \item Evaluate the impacts of the control factors on the S/N ratios and on the LBD values.
    \item For each factor, which has significant impact on the S/N ratios values, we choose the level which increases the S/N ratios.
    \item For the factors, which do not have significant impact on the S/N ratios and have significant impact the LBD values, we choose the level which better approximate the lower bound values.
    \item For the factors, which have significant impact neither on the S/N ratios nor on the LBD values, we select the factor levels regarding the more economic manner, that is, we choose the levels which have less impact on the algorithm running time.
\end{enumerate}

We can see in Table \ref{factors_results} the results after carrying out the computational tests for each trial with the test instance set.

\begin{table}[H]
\label{factors_results}
\caption{LBD, S/N ratio, and average execution time results for each trial}
\centering
\resizebox{\textwidth}{!}{%
\begin{tabular}{cccccccclccr}
\hline
\multirow{2}{*}{\textbf{Trial}} & \multicolumn{7}{c}{\textbf{Control factors}} &  & \multirow{2}{*}{\textbf{\begin{tabular}[c]{@{}c@{}}LBD\\ values\end{tabular}}} & \multirow{2}{*}{\textbf{\begin{tabular}[c]{@{}c@{}}S/N\\ ratios\end{tabular}}} & \multirow{2}{*}{\textbf{\begin{tabular}[c]{@{}c@{}}Average\\ time (s)\end{tabular}}} \\ \cline{2-8}
 & \textbf{\textit{TP}} & \textbf{\textit{NG}} & \textbf{\textit{MUT}} & \textbf{\textit{RST}} & \textbf{\textit{AS}} & \textbf{\textit{CRS}} & \textbf{\textit{TER}} & \textbf{} &  &  &  \\ \cline{1-8} \cline{10-12} 
1 & 1 & 1 & 1 & 2 & 2 & 1 & 1 &  & 0.23719 & -80.01779 & 274.7 \\
2 & 1 & 2 & 3 & 1 & 1 & 2 & 1 &  & 0.28382 & -80.88757 & 149.5 \\
3 & 1 & 2 & 2 & 1 & 2 & 1 & 2 &  & 0.21597 & -79.52707 & 589.7 \\
4 & 1 & 1 & 2 & 2 & 1 & 2 & 2 &  & 0.33001 & -81.74368 & 69.0 \\
5 & 2 & 2 & 2 & 2 & 1 & 1 & 1 &  & 0.20842 & -79.28486 & 245.7 \\
6 & 2 & 1 & 2 & 1 & 2 & 2 & 1 &  & 0.31188 & -81.62437 & 482.4 \\
7 & 2 & 1 & 3 & 1 & 1 & 1 & 2 &  & 0.20360 & -79.18351 & 180.1 \\
8 & 2 & 2 & 1 & 1 & 1 & 2 & 2 &  & 0.32368 & -81.78923 & 255.1 \\
9 & 2 & 2 & 3 & 2 & 2 & 2 & 2 &  & 0.28475 & -80.88293 & 344.0 \\
\hline
\end{tabular}
}
\end{table}

In Figure \ref{fig:main_SN}, we can see the main effects plot for the control factors using S/N ratios as the response variable. In Figure \ref{fig:boxplot_SNR}, we show the boxplots for each factor also using S/N ratios as the response variable. The mean response is clearly influenced by the type of crossover, while it is not so clear to affirm whether or not the other factors influence the response variable.

\begin{figure}[H]
    \centering
    \caption{Main effects plot for S/N ratio for lowerbound deviation values}
    \includegraphics[scale=0.5]{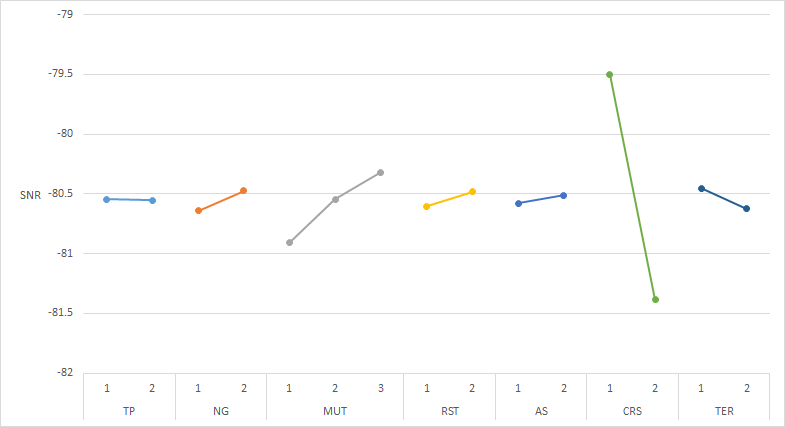}
    \label{fig:main_SN}
\end{figure}

\begin{figure}[H]
    \centering
    \caption{Boxplots for S/N ratio values with each factor}
    \includegraphics[scale = 0.7]{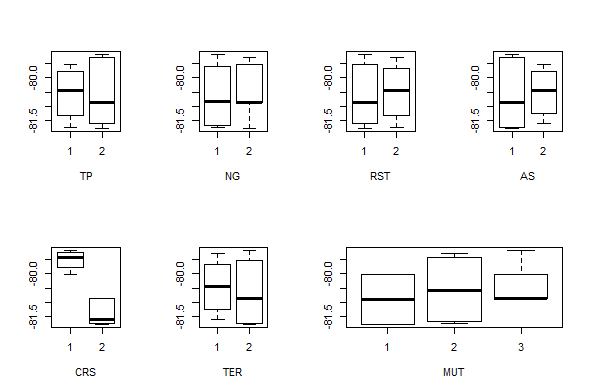}
    \label{fig:boxplot_SNR}
\end{figure}

Adjusting the linear regression model for all seven factors and performing an ANOVA test, we observe that only CRS is statistically significant with $P$-value 0.0259. Then we remove, one by one, the factors whose $P$-value is the greatest and readjust the regression model until all factors are statistically significant obtaining the ANOVA results in Table \ref{anovaSNR_most}.

\begin{table}[H]
\centering
\caption{ANOVA table for S/N ratios for linear regression model fit considering all 7 factors}
\begin{tabular}{lcccrcl}
\cline{1-6}
\textbf{Factor} & \textbf{df} & \textbf{Sum Sq} & \textbf{Mean Sq} & \textbf{$F$ value} & \textbf{$P$-value} &  \\ \cline{1-6}
TP& 1& 0.0002& 0.0002& 0.0125& 0.9290& \\ 
NG& 1& 0.0640& 0.0640& 4.5126& 0.2801& \\ 
MUT& 1& 0.3786& 0.3786& 26.6784& 0.1218& \\ 
RST& 1& 0.0749& 0.0749& 5.2817& 0.2613& \\ 
AS& 1& 0.0292& 0.0292& 2.0585& 0.3875& \\ 
CRS& 1& 8.5564& 8.5564& 602.9946& 0.0259& *\\ 
TER& 1& 0.0196& 0.0196& 1.3807& 0.4489& \\ 
Residuals& 1& 0.0142& 0.0142& & & \\ 
Total& 8& 9.1371& & & & \\ 
\cline{1-6}
Signif. codes: & & ‘*’ 0.05& & & & \\ 
\cline{1-6}
\end{tabular}
\end{table}

\begin{table}[H]
\centering
\caption{ANOVA table for S/N ratio for linear regression model fit considering most significant factors}
\begin{tabular}{lcccrcl}
\cline{1-6}
\textbf{Factor} & \textbf{df} & \textbf{Sum Sq} & \textbf{Mean Sq} & \textbf{$F$ value} & \textbf{$P$-value} &  \\ \cline{1-6}
NG& 1& 0.0627& 0.0627& 5.2218& 0.08431& .\\ 
MUT& 1& 0.3671& 0.3671& 30.5578& 0.00523& **\\ 
RST& 1& 0.0794& 0.0794& 6.6076& 0.06195& .\\ 
CRS& 1& 8.5799& 8.5799& 714.2988& 0.00001& ***\\ 
Residuals& 4& 0.0480& 0.0120& & & \\ 
Total& 8& 9.1371& & & & \\ \cline{1-6}
Signif. codes: & & 0 ‘***’ & ‘**’ 0.01& ‘*’ 0.05&  ‘.’ 0.1& \\  \cline{1-6}
\end{tabular}
\label{anovaSNR_most}
\end{table}

The number of generations, mutation rate, restart, and type of crossover showed to be statistically significant, meaning that we chose the levels whose average S/N ratios are the greater. The parameter levels chosen as a result of the ANOVA test are $1000r$ generations, $0.05$ of mutation rate, $\ceil{0.2r}$ generations with no improvement to apply restart, and crossover type 1.

As regards to the LBD as response variable to the linear regression model. We observe in the main effects plot in Figure \ref{fig:mainE_LBD} and in boxplots in Figure \ref{fig:boxplots_LBD} that LBD have a similar behavior on the control factors. However, we note that, in this case, the lower the LBD value the better.

\begin{figure}[H]
    \centering
    \caption{Main effects plot for lowerbound deviation}
    \includegraphics[scale = 0.]{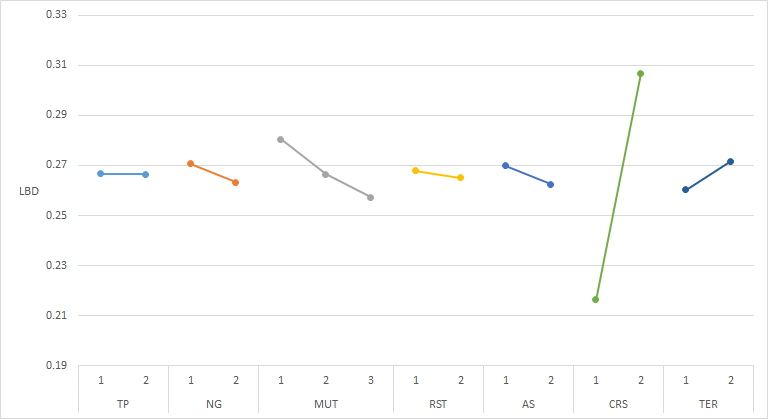}
        \label{fig:mainE_LBD}
\end{figure}

\begin{figure}[H]
    \centering
    \caption{Boxplots for LBD values with each factor}
    \includegraphics[scale = 0.7]{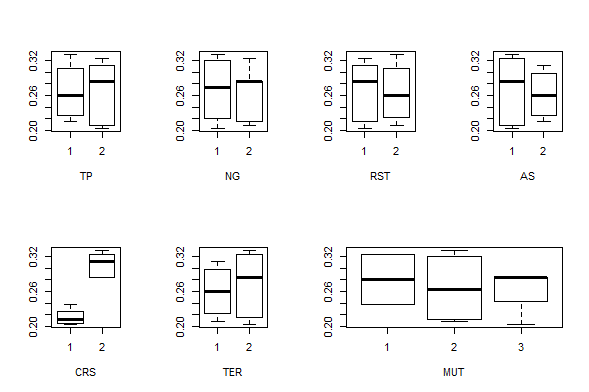}
    \label{fig:boxplots_LBD}
\end{figure}

Adjusting the linear regression model for all the seven factors and performing an ANOVA test using the LDB as response variable, we conclude that only CRS is statistically significant with $P$-value 0.03219. Therefore, we remove from the regression model the variables, one by one, whose $P$-value is the greatest and readjust the model until all factors are statistically significant achieving the ANOVA results illustrated in Table \ref{anovaSNR_most}.

\begin{table}[H]
\centering
\caption{ANOVA table for LBD values for linear regression model fit considering all 7 factors}
\begin{tabular}{lcccrcl}
\cline{1-6}
\textbf{Factors} & \textbf{df} & \textbf{Sum Sq} & \textbf{Mean Sq} & \textbf{$F$ value} & \textbf{$P$-value} &  \\ \cline{1-6}
TP & 1 & 0.00000 & 0.00000 & 0.00340 & 0.96270 &  \\
NG & 1 & 0.00012 & 0.00012 & 2.39350 & 0.36531 &  \\
MUT & 1 & 0.00058 & 0.00058 & 11.48800 & 0.18265 &  \\
RST & 1 & 0.00006 & 0.00006 & 1.23330 & 0.46669 &  \\
AS & 1 & 0.00021 & 0.00021 & 4.22270 & 0.28833 &  \\
CRS & 1 & 0.01960 & 0.01960 & 390.49990 & 0.03219 & * \\
TER & 1 & 0.00011 & 0.00011 & 2.27660 & 0.37261 &  \\
Residuals & 1 & 0.00005 & 0.00005 &  &  &  \\
Total & 8 & 0.02074 &  &  &  &  \\ \cline{1-6}
Signif. codes: &  & ‘*’ 0.05 &  &  &  &  \\ \cline{1-6}
\end{tabular}
\end{table}

\begin{table}[H]
\centering
\caption{ANOVA table for LBD values for linear regression model fit considering most significant factors}
\begin{tabular}{lcccrcl}
\cline{1-6}
\textbf{Factors} & \textbf{df} & \textbf{Sum Sq} & \textbf{Mean Sq} & \textbf{$F$ value} & \textbf{$P$-value} &  \\ \cline{1-6}
MUT	&1	&0.00063	&0.00063	&6.02770	&0.04944 &*\\
CRS	&1	&0.01949	&0.01949	&187.86960	&0.00001 &***\\
Residuals	&6	&0.00062	&0.00010	&	&\\
Total	&8	&0.02074	&	&	&\\
 \cline{1-6}
Signif.   codes:&  & ‘***’ 0 & ‘*’ 0.05 &  &  &  \\ \cline{1-6}
\end{tabular}
\end{table}

Taking into consideration the LBD as response variable to the regression model, only the mutation rate, and type of crossover are statistically significant, meaning that we would choose the mutation rate $0.05$, and crossover type 1. However, these variables were already fixed at the S/N ratios analysis, and no factors that were not statistically significant for the S/N ratios showed to be statistically significant with LBD values. This leads us to choose the levels that would spend less computational time, for the factors whose level was not selected yet. Therefore, the most robust parameterization of the levels for the proposed control factors is: population size 25, $1000 r$ generations, $200r$ generations with no improvement to apply restart, $100 r$ pseudo-random solutions generated in the initial population and restarts, crossover type 1, $5$ preserved individuals upon restart, and mutation rate of $0.05$.

\subsection{Analysis of the final genetic algorithm parameterization}

In order to observe the genetic algorithm behaviour, we run the GA with instance cwp021. Figure \ref{fig:cwp021} illustrates the best fitness and mean fitness of the populations along all generations. The \textit{x-axis} of the Figure \ref{fig:cwp021} is on logarithmic scale. The largest improvement takes place during the first generations of the GA, while in the last ones the best fitness is stagnant with some improvement upon the first restart.

\begin{figure}[H]
    \centering
    \caption{Average and best objective function value curves for instance cwp021 along generations of the selected genetic algorithm parameterization}
    \includegraphics[width = 0.65\linewidth]{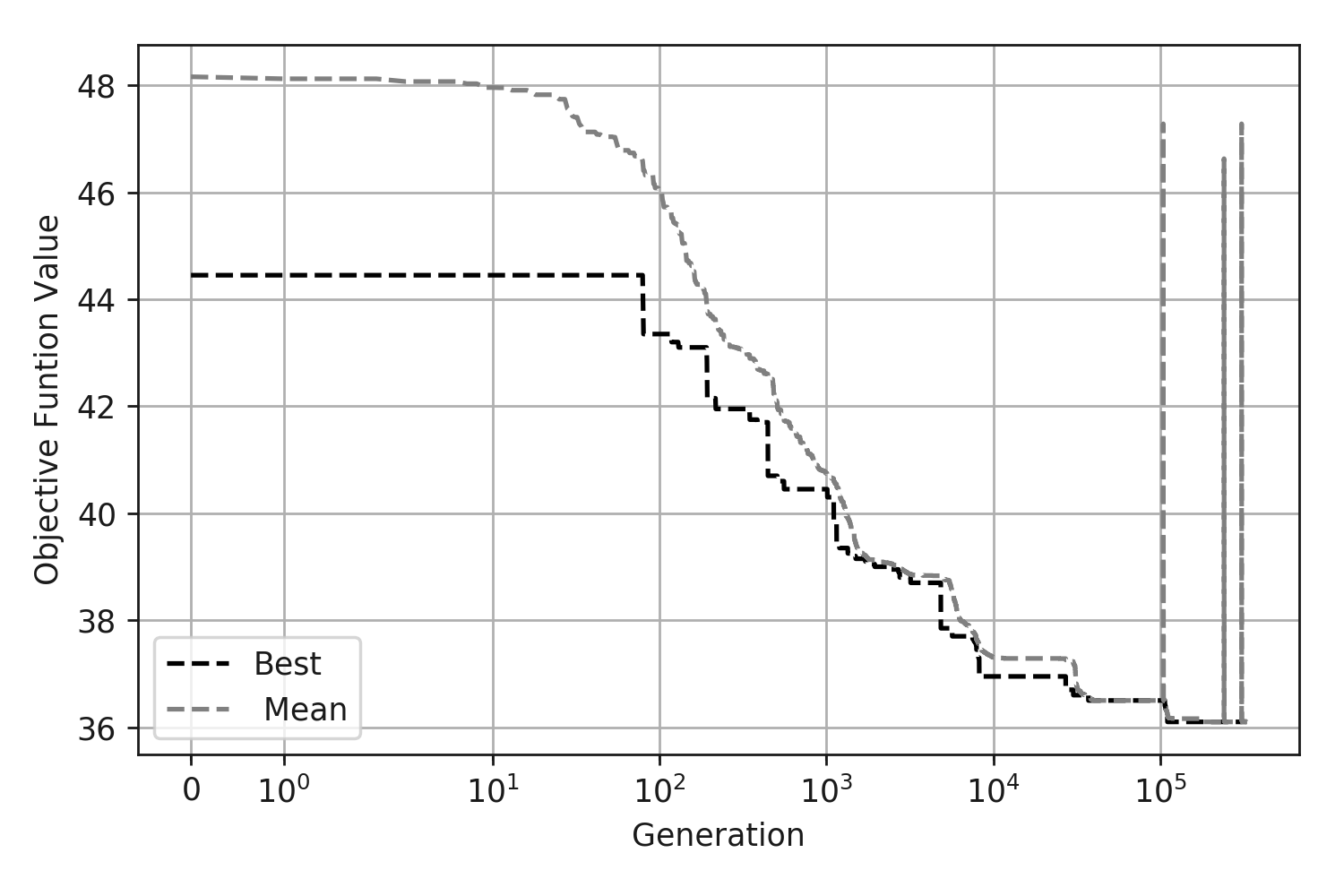}
    \label{fig:cwp021}
\end{figure}

In Figure \ref{fig:LBD_CPLEX_GA}, the best fitness values obtained by running the GA were better than CPLEX in five instances, while solutions were obtained for all instances which CPLEX could not solve.

\begin{figure}[H]
    \centering
    \caption{Lower bound relative deviations for CPLEX and GA with the selected parameterization}
    \includegraphics[width = 0.7\linewidth]{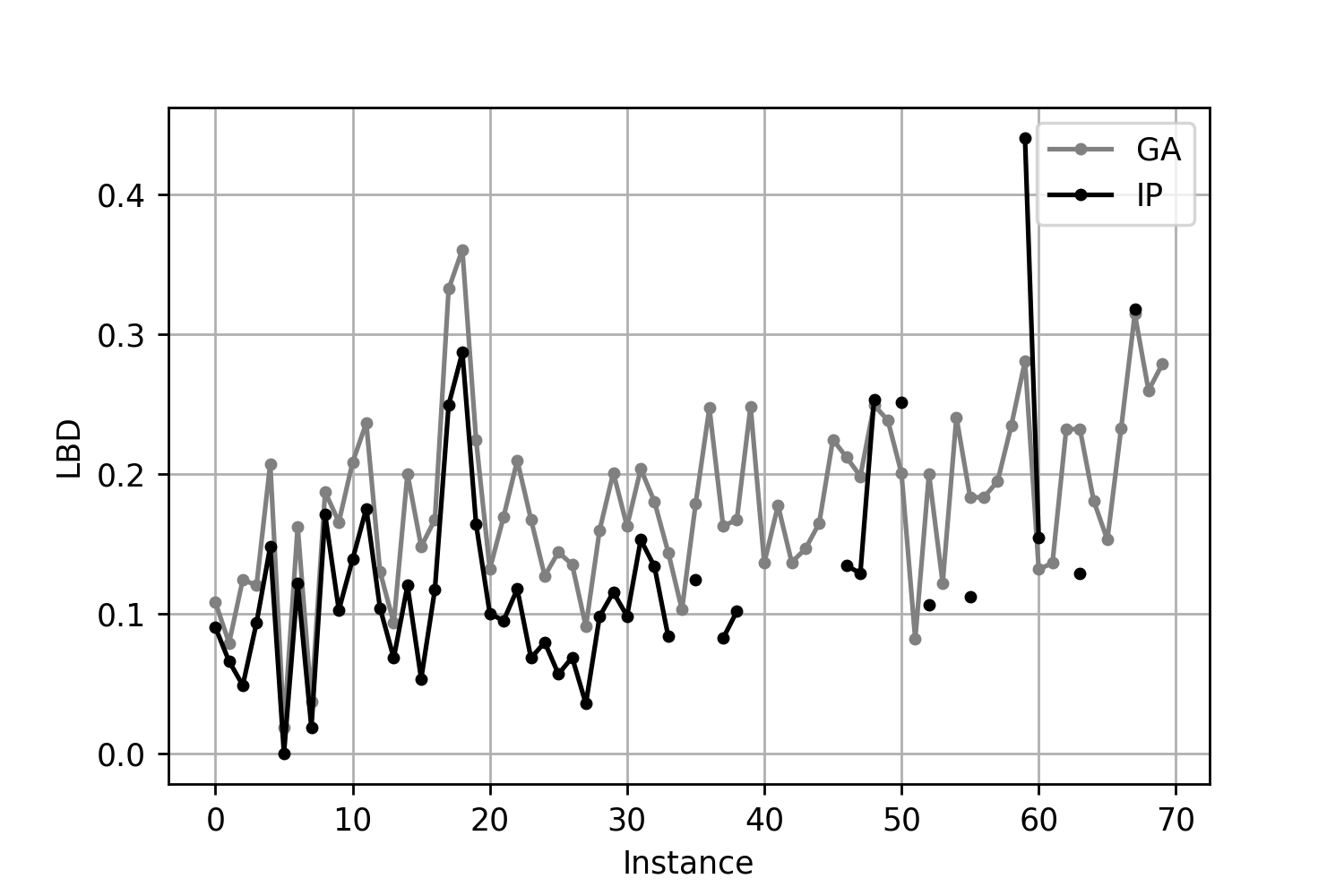}
    \label{fig:LBD_CPLEX_GA}
\end{figure}

In Figures \ref{fig:tempo_ga_cplex} and \ref{fig:tempo_ga_cplex_log}, the time spent by the GA on solving each instance was significantly better than the CPLEX solution time on the large and medium-sized instances. Thus, CPLEX was faster than the GA in the small-sized instances. The \textit{y-axis}  \ref{fig:tempo_ga_cplex_log} is in logarithmic scale.

\begin{figure}[H]
    \centering
    \caption{Mean time for each instance solved by CPLEX and GA with the selected parameterization}
    \includegraphics[width = 0.7\linewidth]{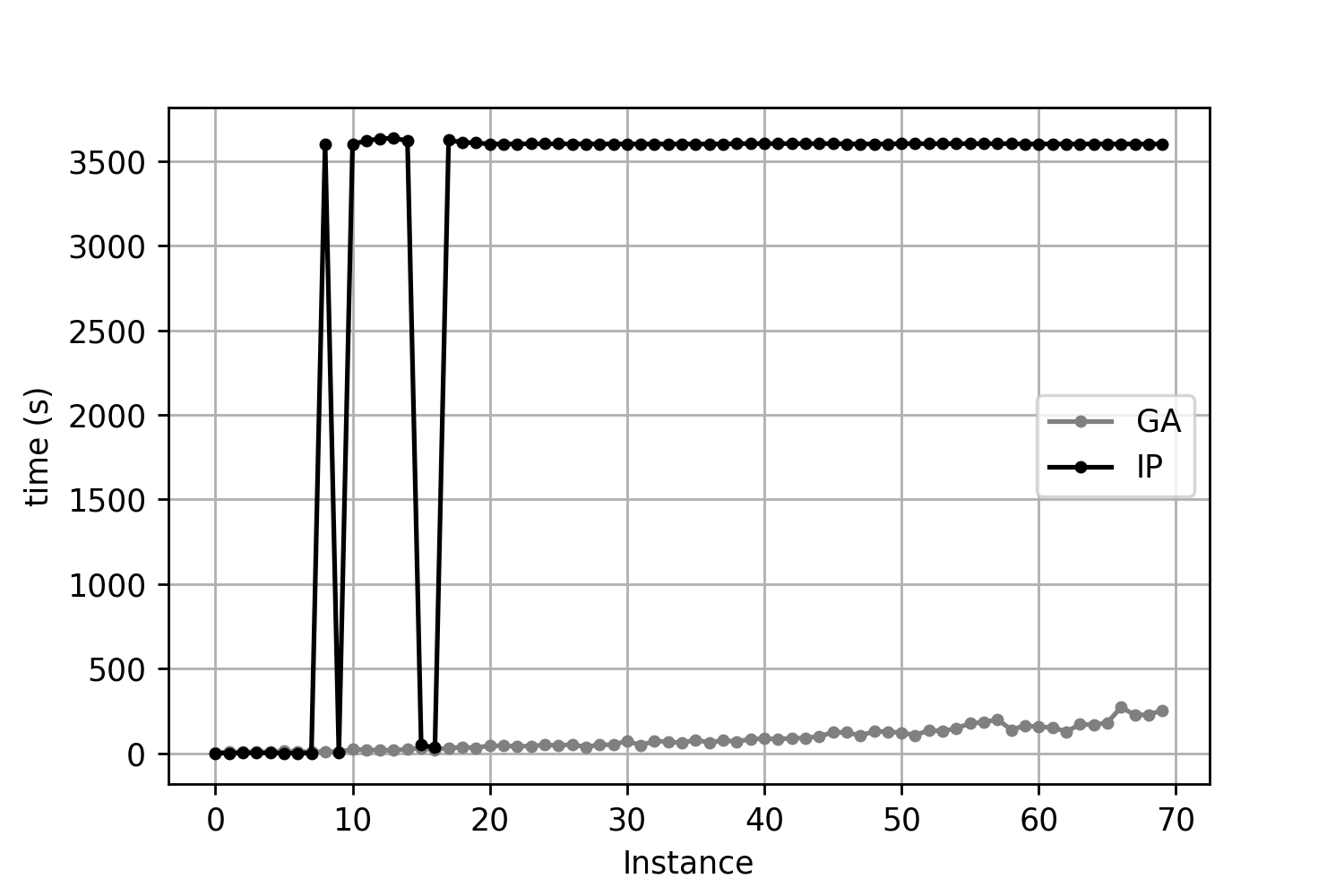}
    \label{fig:tempo_ga_cplex}
\end{figure}

\begin{figure}[H]
    \centering
    \caption{Mean time for each instance solved by CPLEX and GA with the selected parameterization, with y-axis in logarithmic scale}
    \includegraphics[width = 0.7\linewidth]{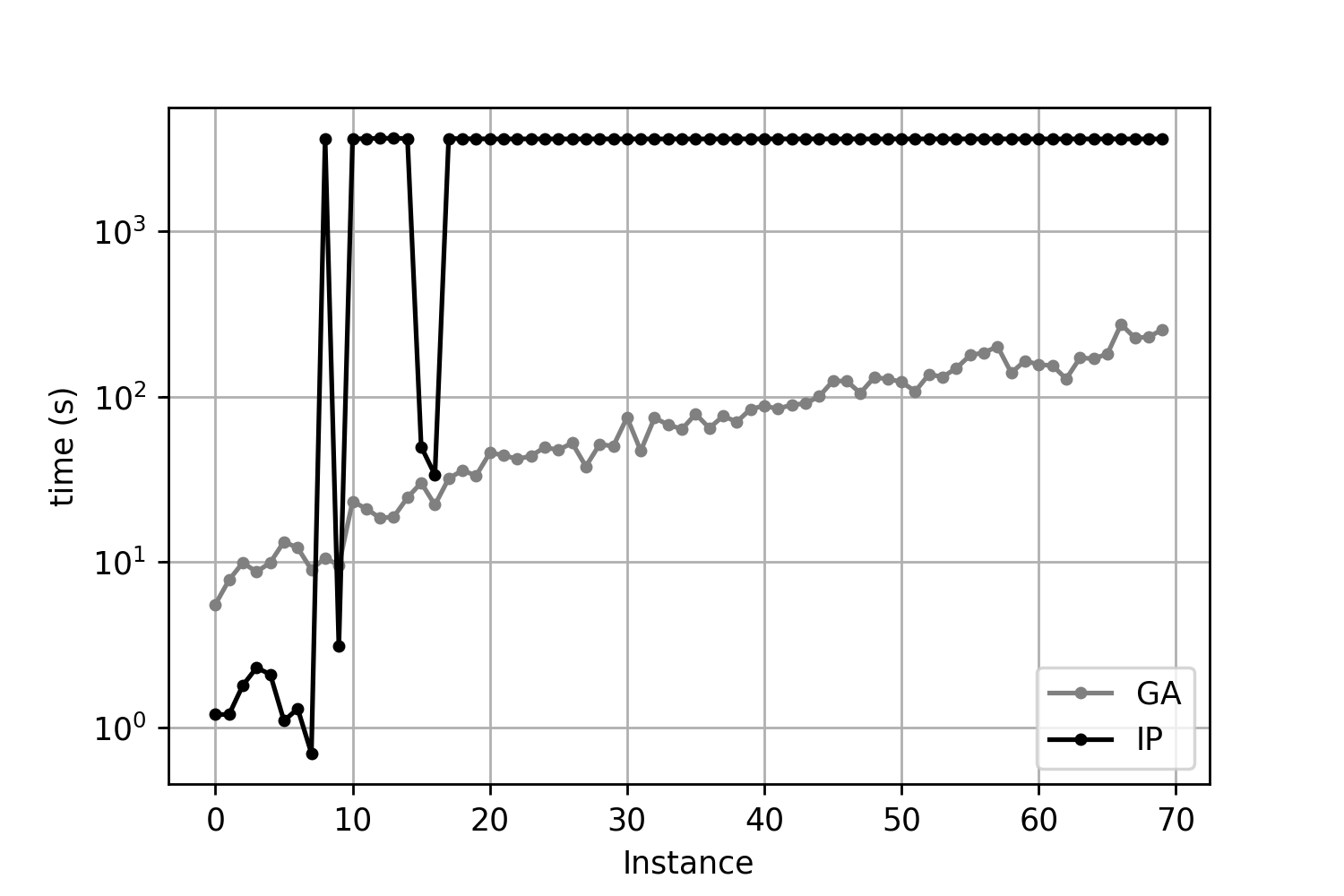}
    \label{fig:tempo_ga_cplex_log}
\end{figure}

\section{Final remarks}
\label{P2:conclusion}

In this work, we proposed a novel variant of cutting sequencing problems called the \textit{integrated cutting and packing heterogeneous precast beams multiperiod production planning} (ICP-HPBMPP), which, to the best of the authors' knowledge, has not yet been studied and may have a large impact on both real-world and theoretical studies. The ICP-HPBMPP consists in integrating the problem of production planning of precast beam with the problem of cutting the traction elements used in such production, while taking into consideration the generation of leftovers and bar generated via overlapping.

We argued that the problem is NP-hard and proposed an integer linear programming model for its solution, in addition to a lower bound on its optimal objective function value. We also showed that restricting the formulation to using exclusively maximal packing patterns does not change the optimal solution set of the problem.

We also proposed three constraint programming models for generating distinct types of beam production patterns. Additionally, we introduced a set of benchmark instances and carried out computational experiments in order to evaluate the relative performance of the different solution methods studied.

The experiments showed that the integer programming model can be used to solve small size instances, while it typically does not reach optimality while solving medium size instances. In addition, the model usually does not find feasible solutions for large size instances. We introduced a genetic algorithm for solving the problem and fine tuned its parameters by means of a \textit{D-optimal} experimental design to achieve improved robustness of the algorithm. 
The final genetic algorithm is an attractive alternative to the integer programming model, resulting in high-quality solutions in shorter solution times as compared with the exact model. 

There are numerous opportunities for future work regarding the ICP-HPBMPP. In the domain of modeling, the problem can be modified to take into consideration distinct types of bars varying in matter of diameter or material, instead of only in matter of length. Also, dynamic demand could be considered, i.e., in each period a new demand of beams could be included, while not exceeding a prescribed stock of bars. Regarding solution approaches, multi-objective optimization algorithms can naturally be applied to the problem, since it involves preferences between makespan and bar waste. Decomposition approaches, such as column generation, or MIP heuristics, e.g., size-reduction heuristics, can also be interesting methods to be explored in conjunction with the proposed integer programming model.

\begin{acknowledgements}
This study was financed in part by the Coordenação de Aperfeiçoamento de Pessoal de Nível Superior - Brasil (CAPES) - Finance Code 001.
\end{acknowledgements}

%
\section*{Conflict of interest}

The authors declare that they have no conflict of interest.

\bibliographystyle{spbasic}      
\bibliography{bibtex}   

\end{document}